\documentclass[11pt]{amsart}
 \usepackage{quiver}
\usepackage{amsmath}
\usepackage{amssymb}
\usepackage{mathrsfs}
\usepackage[all]{xy}
\usepackage{graphicx}
\usepackage{latexsym}
\usepackage{amsxtra}
\usepackage{url}
\usepackage{verbatim}
\usepackage{adjustbox}

\usepackage{tikz}
\usetikzlibrary{shapes.misc, positioning}

\usepackage{float}

\usepackage[urlcolor=blue]{hyperref}

\newtheorem*{thmMain}{Theorem}

\newtheorem{theorem}{Theorem}[section]
\newtheorem{claim}[theorem]{Claim}

\newtheorem{lemma}[theorem]{Lemma}

\newtheorem{proposition}[theorem]{Proposition}
\newtheorem{corollary}[theorem]{Corollary}

\theoremstyle{definition}
\newtheorem{definition}[theorem]{Definition}
\newtheorem{example}[theorem]{Example}

\newtheorem{question}[theorem]{Question}

\theoremstyle{remark}
\newtheorem{remark}[theorem]{Remark}

\newtheorem{fact}[theorem]{Fact}

\def\l{{\langle}}
\def\r{{\rangle}}
\def\forces{\Vdash}
\def\Q{\mathbb Q}
\def\P{\mathbb P}

\def\CI{\mathcal I}
\def\CJ{\mathcal J}
\def\CF{\mathcal F}

\def\kk{\kappa^\kappa}
\def\mfb{\mathfrak{b}}
\def\mfd{\mathfrak{d}}

\def\hook{\upharpoonright}
\newcount\skewfactor
\def\mathunderaccent#1#2 {\let\theaccent#1\skewfactor#2
\mathpalette\putaccentunder}
\def\putaccentunder#1#2{\oalign{$#1#2$\crcr\hidewidth
\vbox to.2ex{\hbox{$#1\skew\skewfactor\theaccent{}$}\vss}\hidewidth}}

\def\smallbox#1{\leavevmode\thinspace\hbox{\vrule\vtop{\vbox
   {\hrule\kern1pt\hbox{\vphantom{\tt/}\thinspace{\tt#1}\thinspace}}
   \kern1pt\hrule}\vrule}\thinspace}

\newcommand{\jbd}[1]{\mathcal{J}^{#1}_{bd}}

\DeclareMathOperator{\GCH}{GCH}
\DeclareMathOperator{\id}{Id}

\DeclareMathOperator{\dom}{dom}

\newcommand{\cf}{{\rm cf}}

\title{Eventual Capture on a Measurable Cardinal}
\author{Tom Benhamou}
\address[Benhamou]{Department of Mathematics, Rutgers University, New Brunswick, 08854-8019 NJ, USA}
\email{tom.benhamou@rutgers.edu}
\author{Corey Bacal Switzer}
\address[Switzer]{Institute of Mathematics, Univeristy of Vienna, Kolingasse 14-16, 1090 Vienna, Austria}
\email{corey.bacal.switzer@univie.ac.at}

\thanks{
\emph{Acknowledgments.} The first author is supported by the NSF under Grant
No. DMS-2346680.
The research of the second author was funded in whole or in part by the Austrian Science Fund (FWF) grant doi 10.55776/ESP548. For open access purposes, the author has applied a CC BY public copyright license to any author-accepted manuscript version arising from this submission.
}
\date{}

\begin{document}

\begin{abstract}
    We continue the study from \cite{BrendleFreidmanMontoya, vandervlugtlocalizationcardinals} of localization cardinals $\mfb_\kappa(\in^*)$ and $\mfd_\kappa(\in^*)$ and their variants at regular uncountable $\kappa$. We prove that if $\kappa$ is measurable then these cardinals trivialize. We also provide other fundamental restrictions in the most general setting. We prove the results are optimal by forcing different values for $\mathfrak{b}_{\id^+}(\in^*),\mathfrak{d}_{\id^{++}}(\in^*)$ at a measurable. As a by-product, we prove the consistency of $\mfb_h(\in^*) < \mfb_{h'}(\in^*)$ for functions $h, h' \in \kk$, thus answering a question of Brendle, Brooke-Taylor, Friedman and  Montoya. Moreover, we study the relation between these cardinals and other well-known cardinal invariants.
    
    
\end{abstract}

\maketitle

\section{Introduction}
The relation $\in^*$ of eventual capture, also known as localization, was introduced by Bartoszy\'nski in \cite{BartoSlaloms} to provide a combinatorial perspective on the ideal of Lebesgue measure zero subsets of reals. The key step in his celebrated theorem from \cite{addNimpliesaddM} that $\mathrm{add}(\mathcal N) \leq \mathrm{add}(\mathcal M)$ is the fact that\footnote{Precise definitions of the cardinals here will be defined below.} $\mfb (\in^*) = \mfb_h(\in^*) = \mathrm{add}(\mathcal N)$ and dually $\mfd (\in^*) = \mfd_h (\in^*) = \mathrm{cof}(\mathcal N)$ for every strictly increasing function $h:\omega \to \omega$ tending to infinity. Since then, localization has appeared consistently in the literature of set theory of the reals. See, e.g. \cite{BarJu95}, for more details and history.

In the absence of an obvious null ideal on the generalized Baire space $\kappa^\kappa$, the authors of \cite{BrendleFreidmanMontoya} studied the natural analogues of $\mfb_\kappa(\in^*)$ and $\mfd_\kappa(\in^*)$ as stand-ins for $\mathrm{add}(\mathcal N)$ and $\mathrm{cof}(\mathcal N)$ in the higher Cicho\'n diagram for an inaccessible cardinal $\kappa$. They showed that these cardinals satisfy roughly the same collection of facts as their natural analogue on $\omega$ with a few key differences. This analysis was furthered by van der Vlugt in \cite{vandervlugtlocalizationcardinals}. The purpose of this paper is study these cardinals further as well as some natural variants with a particular emphasis on the case where $\kappa$ is measurable. 

Starting with the work of Cummings and Shelah \cite{CUMMINGSShelah}, cardinal invariants on regular uncountable cardinals were studied, and recently have been of interest \cite{BrendleFreidmanMontoya,BROOKETAYLOR201737, EskewFischer, RaghavanShelah} due to several applications e.g. to ultrafilter theory \cite{tomCohesive,TomGabeMeasures} as well as explicating how the higher Baire and Cantor spaces differ from their classical counterparts. Indeed there are fundamental discrepancies between cardinal characteristics of the continuum and above the continuum. The larger the cardinal is, it appears that more $\mathsf{ZFC}$ facts can be proven. For example, while the pairs $\mathfrak{b},\mathfrak{s}$ and $\mathfrak{d},\mathfrak{r}$ are independent invariants on $\omega$, Raghavan and Shelah \cite{RaghavanShelah,Raghavan/Shelah2019} proved that for every regular uncountable cardinal $\kappa$, $\mathfrak{s}_\kappa\leq\mathfrak{b}_\kappa$ and that $\mathfrak{d}_\kappa\leq \mathfrak{r}_\kappa$ when $\kappa>\beth_{\omega}$. For $\omega<\kappa<\beth_\omega$ the question is still open. Other theorems of this kind include recent result of Benhamou and Goldberg \cite{TomGabeMeasures}, that if there is a simple $P_\lambda$-points then $\mathfrak{b}_\kappa=\mathfrak{d}_\kappa=\lambda$ (and thus such $\lambda$ is unique). This is in contrast to the result of Br\"{a}uninger and Mildenberger \cite{MildenBraun} that there may be a $P_{\aleph_1}$- point and $P_{\aleph_2}$-point on $\omega$ simultaneously. A more detailed look at more provable inequalities was undertaken in \cite{FischerSoukup}.

Moving to larger cardinals, further discrepancies appear. For example, Suzuki showed in \cite{suzuki} that $\mathfrak{s}_\kappa > \kappa$ is equivalent to $\kappa$ being weakly compact. Zapletal \cite{ZapletalSplitting} and Ben-Neria Gitik \cite{OmerMoti} showed that $\mathfrak{s}(\kappa)>\kappa^+$ is equiconsistent with $o(\kappa)=\kappa^{++}$. Ben-Neria and Gitik asked whether having $\mathfrak{s}(\kappa)>\kappa^+$ while $\kappa$ is measurable has the same consistency strength, which turned out to be false. Recent work of Benhamou and Ben-Neria \cite{TomOmer} shows that preserving measurability with $\mathfrak{s}(\kappa)>\kappa^+$ is equiconsistent with $o(\kappa)=\kappa^{++}+1$.
The moral of these results is that preserving measurability sometimes requires extra large cardinal effort and that the presence of strong ultrafilters effects cardinal characteristics.

In this paper, we study the invariants $\mathfrak{b}_h(\in^\CI)$ and $\mathfrak{d}_h(\in^{\CI})$ where $h:\kappa\to\kappa$ and $\CI$ is an ideal on $\kappa$ (see Definition~\ref{def: localization invariants}) with a particular emphasis on the case where $\kappa$ is measurable. Here, the differences become amplified, as both new $\mathsf{ZFC}$ results appear whose analogues on $\omega$ and smaller regular uncountable cardinals are independent and new independence results appear whose analogues on $\omega$ are $\mathsf{ZFC}$-results. Our first main result on $\mfb_\kappa (\in^*)$ and $\mfd_\kappa(\in^*)$ for $\kappa$ measurable represents a split with what should be expected for higher cardinal invariants and this split is the central point of interest of this paper: 
\begin{thmMain}
    Suppose that $\kappa$ is a measurable cardinal, then $\mathfrak{b}_{\kappa}(\in^*)=\kappa^+$ and $\mathfrak{d}_{\kappa}(\in^*)=2^\kappa$. Moreover, $\mathfrak{d}_{\id^{+}}(\in^*)=2^\kappa$, where $\id^+$ is the function mapping $\alpha$ to $\alpha^+$. 
\end{thmMain} 
 Denote by $\id^{+\xi}(\alpha)=\alpha^{+\xi}$. We also provide further restriction on the size of the cardinals $\mathfrak{b}_{\id^{+\xi}}(\in^*)$ and $\mfd_{\id^{+\xi}}(\in^*)$.
\begin{thmMain}
    Suppose that $\kappa$ is a measurable cardinal, then 
    $$\mathfrak{b}_{\id^{+\xi}}(\in^*)\leq \kappa^{+\xi+1}.$$
and 
$$\mathfrak{d}_{\id^{+\xi}}(\in^*)<2^\kappa\Rightarrow 2^\kappa\leq\kappa^{+\xi}.$$
\end{thmMain}
Our results generalize to many ideals $\CI$; any ideal that can be extended to a normal prime ideal such as the non-stationary ideal. Moreover,
we partially bring these result down to smaller large cardinals such as completely ineffable cardinals (see Corollary~\ref{Cor: completely ineff}).

To show the optimality of our result, we provide several consistency results. We prove the consistency of $\mathfrak{b}_{\id^+}(\in^*)>\kappa^+$ and $\mathfrak{d}_{\id^{++}}(\in^*)<2^\kappa$ at a measurable cardinal:
\begin{thmMain}
    Relative to a supercompact cardinal, it is consistent for a supercompact cardinal $\kappa$ that $\kappa^+<\mathfrak{b}_{\id^+}(\in^*)$. It is also consistent that $\mathfrak{d}_{\id^{++}}(\in^*)<2^\kappa$. 
\end{thmMain}

In particular, the above provides an answer to \cite[Question 71]{BrendleFreidmanMontoya}, also asked as \cite[Question 4.3]{vandervlugtlocalizationcardinals}, since in the above model $\kappa$ is measurable and therefore $\mathfrak{b}_\kappa(\in^*)=\kappa^+<\mathfrak{b}_{\id^{+}}(\in^*)$.

The above theorems work equally well for the variants $\mfb_h(\in^{cl})$ and $\mfd_h(\in^{cl})$ where capture is only insisted upon on a club. Nevertheless we show that, regardless of whether $\kappa$ satisfies a large cardinal property these ``club variants" are distinct.
\begin{thmMain}
    For any regular uncountable cardinal and any strictly increasing function $h:\kappa \to \kappa$ it is (separately) consistent that $\mfb_h (\in^*) < \mfb_h (\in^{cl})$ and $\mfd_h(\in^{cl}) < \mfd_h(\in^*)$. Moreover, relative to the consistency of a supercompact cardinal, if $h$ is the power function then the above are consistent with $\kappa$ being supercompact. 
\end{thmMain}
Note that, by contrast, $\mfb_\kappa$ and $\mfd_\kappa$ are equivalent to their club variants, see \cite{CUMMINGSShelah}\footnote{For $\mfb_\kappa$ this is true always while for $\mfd_\kappa$ it is only known above $\beth_\omega$ and open below, see\cite[Theorem 8]{CUMMINGSShelah}.}. Similarly we show that $\mfb_\kappa(\in^*)$ is independent of $\mathfrak{s}_\kappa$ - thus showing that the analogue of the Raghavan-Shelah result mentioned above also fails.

\begin{thmMain}
  Relative to the existence of a supercompact cardinal both strict inequalities between $\mfb_h(\in^*)$ and $\mathfrak{s}_\kappa$ are consistent for any $h:\kappa \to \kappa$. If $h$ is the power function or $\id^+$ then they are moreover consistent with $\kappa$ remaining supercompact. 
\end{thmMain}

Regarding the consistency strength of the previous results on measurable cardinals, 
in \S\ref{Section: Coverling} we discuss certain covering properties that an ultrapower must satisfy in order to allow a large $\mathfrak{b}$. In particular, the ultrapower should compute certain cofinalities correct and should be correct about at least one cardinal greater than $\kappa^+$.  Another property the ultrapower should have is that its $cf^V(j_U(\kappa))$ must be quite large. We investigate the possibilities of what $cf^V(j_U(\kappa))$ can be and use a previous construction of Gitik to obtain models where $cf^V(j_U(\kappa))$ depends on the ultrafilter from optimal assumptions.

The rest of this paper is organized as follows. In the next section we introduce the localization cardinals in full generality and discuss some other preliminaries. In the following section we examine $\mathsf{ZFC}$-provable relations on the localization cardinals under large cardinal assumptions. In \S\ref{forcing} introduces the forcing notions used to manipulate the localization cardinals and uses these to prove that the $\mathsf{ZFC}$ results are optimal at a measurable by manipulating the localization cardinals at a supercompact. In \S\ref{section: club localization} we consider a variant of the localization forcing for adding slaloms which capture on a club and shows that the associated club variants of $\mfb_h(\in^*)$ and $\mfd_h(\in^*)$ are different from their ``eventual" counterparts. In this section we also separate both types of capturing cardinals from $\mathfrak{s}_\kappa$ and $\mathfrak{p}_\kappa$ and show all of these separations can moreover be obtained alongside large cardinal preservation.  The final section finishes by discussing the aforementioned covering properties. Throughout our notation is standard, conforming e.g. to the monograph \cite{Jech2003}. The reader is referred there for any undefined terms or notation. 
\section{Preliminaries}
Let $\kappa\geq\omega$ be a cardinal. We denote by $\id$, the identity function, where the domain of $\id$ should be adapted to the context. For a set $X$ and a cardinal $\lambda$ we let $[X]^\lambda$ denote the collection of subsets of $X$ of cardinality $\lambda$. Also, $\jbd{\kappa}$ denotes the bounded ideal on $\kappa$; the ideal generated by the set of ordinals $\kappa$. We will always assume that functions $h:\kappa\to\kappa$ below never output finite numbers.
\begin{definition}\label{def: slalom}
    Let  $h:\kappa\to\kappa$ be a function, and let $\CI$ be an ideal over $\kappa$.
    \begin{enumerate}
        
        \item A partial $(h,\CI)$-slalom is a partial function $\varphi:\kappa\to P(\kappa)$ such that $\dom(\varphi)\in \CI^+$ and for every $\alpha\in \dom(\varphi)$, $\varphi(\alpha)\in[\kappa]^{|h(\alpha)|}$.
        \item An $(h,\CI)$-slalom is a partial $h$-slalom with $\dom(\varphi)\in \CI^*$.
        \item A (partial) $h$-slalom  is a  (partial) $(h,\jbd{\kappa})$-slalom.
        \item A (partial) $\kappa$-slalom is an (partial) $\id$-slalom.
    \end{enumerate}
\end{definition} 
Set ($pSL^\CI_h$) $SL^{\CI}_h$ to be the set of (partial) $(h,\CI)$-slaloms $\varphi$. Again, drop the superscript $\CI$ for $\CI=\jbd{\kappa}$ and replace the subscript $h$ with $\kappa$ when $h=\id$. Clearly $SL_h\subseteq pSL^\CI_h$. 
\begin{definition}
    We define the binary relation $\in ^\CI\subseteq \kappa^{\kappa}\times pSL^\CI_h$ by $f\in^\CI \varphi$ if and only if $\{\alpha\in \dom(\varphi)\mid f(\alpha)\notin \varphi(\alpha)\}\in \CI$.  
\end{definition}
For $\CI=\jbd{\kappa}$, we denote $\in^\CI$ by $\in^*$.
As with every binary relation, notions of bounding and dominating numbers are induced: 
\begin{definition}\label{def: localization invariants}
Let $h:\kappa\to\kappa$ be a function and $\CI$ an ideal over $\kappa$. Set
$$\mathfrak{b}_h(\in^\CI)=\min\{|\CF|\mid \CF\subseteq \kappa^\kappa, \ \forall \varphi\in SL_h \exists f\in\CF,\neg(f\in^\CI \phi)\}$$
$$\mathfrak{b}_h(\mathrm{p}\in^\CI)=\min\{|\CF|\mid \CF\subseteq \kappa^\kappa, \ \forall \varphi\in pSL_h \exists f\in\CF,\neg(f\in^\CI \phi)\}$$
$$\mathfrak{d}_h(\in^\CI)=\min\{|\mathcal{A}|\mid \mathcal{A}\subseteq SL_h, \ \forall f\in\kappa^\kappa \exists \varphi\in\mathcal{A}, \ f\in^\CI \varphi\}$$
$$\mathfrak{d}_h(\mathrm{p}\in^\CI)=\min\{|\mathcal{A}|\mid \mathcal{A}\subseteq pSL_h, \ \forall f\in\kappa^\kappa \exists \varphi\in\mathcal{A}, \ f\in^\CI \varphi\}$$
\end{definition}

Again, for $h=\id$ we replace the subscript $h$ by $\kappa$. For example, we denote $\mathfrak{b}_{\id}(\in^*)=\mathfrak{b}_\kappa(\in^*)$, and $\mathfrak{d}_{\id}(p\in^\CI)=\mathfrak{d}_\kappa(p\in^\CI)$. There are obvious Tukey-Galois connections  here\footnote{For more background about Tukey-Galois connections see \cite[\S4]{Blass2010}. Our notations follow \cite[Def. 6\&8]{BrendleFreidmanMontoya}.}:
\begin{proposition}
Let $\CI$ be an ideal extending $\jbd{\kappa}$ and $h:\kappa\to\kappa$ be any function. Then, $$(\kappa^\kappa,SL_h,\in^\CI)\preceq(\kappa^\kappa,pSL^\CI_n,\in^\CI)\preceq (\kappa^\kappa,\kappa^\kappa,\leq^*).$$ In particular:
    \begin{enumerate}
        \item $\mathfrak{b}_{h}(\in^\CI)\leq \mathfrak{b}_h(\mathrm{p}\in^\CI)\leq\mathfrak{b}_\kappa$.
        \item $\mathfrak{d}_{h}(\in^\CI)\geq \mathfrak{d}_h(\mathrm{p}\in^\CI)\geq\mathfrak{d}_\kappa$
    \end{enumerate}
\end{proposition}
For $f,g:\kappa\to\kappa$ denote by $f\leq^{\CI}g$ if $\{\alpha<\kappa\mid f(\alpha)>g(\alpha)\}\in\CI$. By the next fact, these are indeed cardinal characteristics:
\begin{fact}
    Suppose $\jbd{\kappa}\subseteq\CI$ and $h:\kappa\to\kappa$ is such that for $h$ is $\leq_\CI$-above all the constant functions. Then $\kappa^+\leq\mathfrak{b}_h(\in^\CI)\leq \mathfrak{d}_h(\in^\CI)\leq2^\kappa$.
\end{fact}
We have monotonicity in the parameters $h$ and $\CI$:
\begin{proposition}
    Suppose that $\CI\subseteq \CJ$ and $h\leq^\CI t$, then $$(\kappa^\kappa,SL_h,\in^\CJ)\preceq (\kappa^\kappa,SL_t,\in^\CI).$$ In particular,
$\mathfrak{d}_t(\in^\CI)\leq \mathfrak{d}_h(\in^\CJ)$ and $\mathfrak{b}_h(\in^\CJ)\leq\mathfrak{b}_t(\in^\CI)$.

\end{proposition}
These cardinals have another, slightly different characterization as well which is sometimes useful. Given $\kappa$, $h$ and $\CI$ as above and two $(\CI,h)$-slaloms $\varphi$ and $\psi$, let us write that $\varphi \subseteq^\CI \psi$ if and only if $\{\alpha \in \kappa\mid \varphi (\alpha) \nsubseteq \psi(\alpha)\} \in \CI$. We then get, as usual, two cardinals $\mfb_h (\subseteq^\CI)$, the least size of a family of $(\CI,h)$-slaloms with no single $\subseteq^\CI$-bound and $\mfd_h(\subseteq^\CI)$, the least size of a family of $(\CI,h)$-slaloms needed to $\subseteq^\CI$-dominate every $(\CI,h)$-slalom. 

\begin{proposition}\label{subseteqcharacterization}
    For any $\kappa$ with $\kappa^{<\kappa} = \kappa$, any ideal $\CI$ on $\kappa$ extending the bounded ideal and any funcion $h:\kappa \to \kappa$,
    $$(\kappa^\kappa,SL_h,\in^{\CI})\equiv(SL_h,SL_h,\subseteq^{\CI}).$$ In particular, $\mfb_h (\subseteq^\CI) = \mfb_h(\in^\CI)$ and
         $\mfd_h (\subseteq^\CI) = \mfd_h(\in^\CI)$.
\end{proposition}
\begin{proof}
    Given $f:\kappa\to\kappa$ we define $\varphi_f(\alpha)=\{f_i(\alpha)\}$, and note that if $\varphi_f\subseteq^{\CI}\varphi$ then $f\in^{\CI}\varphi$. This shows $(SL_h,SL_h,\subseteq^{\CI})\preceq (\kappa^\kappa,SL_h,\in^{\CI})$. In the other direction, since $\kappa^{<\kappa} = \kappa$ we can fix for each infinite $\alpha < \kappa$ an enumeration of $[\kappa]^{h(\alpha)}$ in order type $\kappa$, say $\{x^\alpha_i\mid i \in \kappa\}$. 
    Now given an $(\CI,h)$-slalom, $\psi$ define $f_{\psi}(\alpha) = \xi$ if and only if $x^\alpha_\xi = \psi(\alpha)$.
    Given an $(\CI,h)$-slalom $\varphi$, let $\psi_\varphi$ be the $(\CI,h)$-slalom defined by $\psi_\varphi (\alpha) = \bigcup_{i\in \varphi(\alpha)}x^\alpha_{i}$ and note that since $|\varphi(\alpha)|=|h(\alpha)|$, $|\psi_\varphi(\alpha)|=|h(\alpha)|$. Verifying that $(\psi\mapsto f_\psi,\varphi\to\psi_{\varphi})$ is a Tukey-Galois connection, suppose that $f_{\psi}\in^{\CI}\varphi$, then for each $\alpha$ such that $f_{\psi}(\alpha)\in \varphi(\alpha)$, we will have by definition that $x^\alpha_{f_{\psi}(\alpha)}\subseteq \psi_{\varphi}(\alpha)$. By definition of $f_{\psi}(\alpha)$, $x^\alpha_{f_{\psi}(\alpha)}=\psi(\alpha)$ and therefore $\psi(\alpha)\subseteq \psi_{\varphi}(\alpha)$, as wanted. 
\end{proof}
The following holds in the case of $\omega$ because of its connection with the null ideal but can be proved combinatorially using the lemma above\footnote{The basic case of below  - that $\mfb_\kappa(\in^*)$ is regular was first pointed out to us by J\"{o}rg Brendle and we thank him for sharing it with us.}.
\begin{proposition}
    For any $\kappa$ with $\kappa^{<\kappa} = \kappa$, any ideal $\CI$ on $\kappa$ extending the bounded ideal and any $h:\kappa \to \kappa$ we have that $\mfb_h(\in^\CI)$ is regular and $\mfd_h(\in^\CI)$ has cofinality at least $\mfb_h(\in^\CI)$.
\end{proposition}

\begin{proof}
   By Lemma \ref{subseteqcharacterization} it suffices to consider the cardinal invariants for $\subseteq^\CI$. This relation is transitive, and for transitive orders the proof is almost identical to the usual analogous fact about $\mathfrak{b}$ and $\mathfrak{d}$.
\end{proof}

One specific case we will be interested in is when $\CI=NS_\kappa$ is the nonstationary ideal. In this case we denote by $f\in^{cl}\phi$ the variation of $f \in^{NS_\kappa} \phi$ i.e. $f \in^{cl}\phi$ if and only if $\{\alpha \mid f(\alpha) \in \phi(\alpha)\}$ contains a club. It will be convenient in such cases to treat slaloms as having a domain restricted to a fixed club. This makes no difference as the next lemma shows.
\begin{lemma}
    Let $\kappa$ be a regular uncountable cardinal, $\lambda \leq 2^\kappa$, and  $h:\kappa \to \kappa$ a strictly increasing function. The following are equivalent:

    \begin{enumerate}
        \item For every family $\mathcal F \subseteq \kk$ of size ${\leq}\lambda$ there is an $(cl,h)$-slalom $\varphi$ so that $f \in^{cl}\varphi$ for all $f \in \mathcal F$.
        
        \item For every family $\mathcal F \subseteq \kk$ of size ${\leq}\lambda$ there is a partial $(cl,h)$-slalom $\phi$ so that $f \in^*\phi$ for all $f \in \mathcal F$ in the sense that $\{\alpha \in \dom(\phi)\mid f(\alpha) \notin \phi (\alpha)\}$ is a bounded.
    \end{enumerate}
\end{lemma}

\begin{proof}
   Let $\mathcal F \subseteq \kk$ be of size $\lambda$, say $\{f_\alpha \mid \alpha \in \lambda\} = \mathcal F$. Since the tail of a club is a club, we can complete any $h$-slalom as in (2) arbitrarily to obtain a slalom as in (1). Conversely, assume (1) holds and let $\phi$ be an $h$-slalom so that $f_\alpha \in^{cl} \phi$ for all $\alpha < \lambda$. Let $C_\alpha$ be the club of points on which $f_\alpha$ is caught by $\phi$. First observe $(1)$ implies that $\mfb_\kappa > \lambda$, indeed, this follows from the aforementioned fact that $\mfb_\kappa$ equals its club version. Now recall as discussed in \cite{FischerGaps} that $\mfb_\kappa$ is also the least size of a family of clubs on $\kappa$ with no pseudo-intersection. Therefore we can find a single club $C \subseteq^*C_\alpha$ for all $\alpha < \lambda$ and hence $\phi\hook C$ is as desired. 
\end{proof}

The following fact justifies the reason we shall only be interested in limit cardinals.
    \begin{fact}\label{fact: Successor}
        Suppose that $\kappa=\lambda^+$, then for every $h:\kappa\to\kappa$ which is $\leq^{\CI}$-above the constant function $\lambda$, we have that $\mathfrak{b}_h(\in^\CI)=\mathfrak{b}_h(\mathrm{p}\in^\CI)=\mathfrak{b}_\kappa(\in^\CI) = \mfb_\kappa$ and $\mathfrak{d}_h(\in^\CI)=\mathfrak{d}_h(\mathrm{p}\in^\CI)=\mathfrak{d}_\kappa(\in^\CI) = \mfd_\kappa$.
    \end{fact}
    The case of singular cardinals seems to be interesting, but it is left for further research. Hence in this paper, we will be focused on inaccessible cardinals. 

    Regarding $\mfb_\kappa(\mathsf{p}\in^*)$ we have the following lower bound.
\begin{proposition}
    For any uncountable regular cardinal $\kappa$ we have $\mathfrak{p}_\kappa \leq \mathfrak{b}_\kappa(\mathsf{p}\in^*)$
\end{proposition}

\begin{proof}
    If $\kappa$ is a successor cardinal then $\mathfrak{b}_\kappa(\mathsf{p}\in^*) =\mfb_\kappa$ by Fact \ref{fact: Successor} above. It is well known, see e.g. \cite{FischerGaps} that $\mathfrak{p}_\kappa \leq \mfb_\kappa$. Therefore we assume that $\kappa$ is inaccessible. By \cite[Lemma 58]{BrendleFreidmanMontoya} there is a ${<}\kappa$-closed, $\kappa$-centered forcing notion with canonical lower bounds (see Definition \ref{clb} below) which generically adds a partial slalom eventually capturing all of the ground model functions from $\kk$. In particular, for any $\lambda$ one can capture any particular $\lambda$ many functions by a generic meeting $\lambda + \kappa$ dense sets. By \cite[Theorem 1.8]{FischerGaps} if $\P$ is a ${<}\kappa$-closed, $\kappa$-centered forcing with canonical lower bounds then there is always a generic filter meeting any ${<}\mathfrak{p}_\kappa$-many dense sets. Putting these facts all together we get that for any ${<}\mathfrak{p}_\kappa$-many functions $f \in \kk$ there is a partial slalom capturing all of them hence $\mathfrak{p}_\kappa \leq \mfb_\kappa (\mathsf{p}\in^*)$.
\end{proof}
We will show below that this proposition is false if $\mfb_\kappa(\mathsf{p}\in^*)$ is replaced by $\mfb_\kappa(\in^{cl})$ or $\mfb_\kappa(\in^*)$.  

    Finally, one last preliminary concerns the large cardinal notions we will need in this paper, which will provide us with a special kind of elementary embedding. Let us set up some notations here, given a transitive model $M$, and a set $X\in M$, we say that $U$ is an $M$-ultrafilter over $X$ if $(M,U)\models U\text{ is an ultrafilter over }X$. Given such an ultrafilter we can form the ultrapower of $M$ by $U$, denoted by $M_U$, by considering all equivalence classes $[f]_U$ for $f:X\to 
    M\in M$. The ultrapower embedding $j_U:M\to M_U$ is given by $j_U(x)=[c_x]_U$, where $c_x$ is the constant function with value $x$. Whenever $M_U$ is well-founded (or has a well founded part) we identify it with its transitive collapse. Recall that {\L}os' theorem says that for any formula $\psi(x_1,...,x_n)$, and any functions $f_1,...,f_n:X\to M\in M$
    $$\{x\in X\mid M\models \psi(f_1(x),...f_n(x))\}\in U\Leftrightarrow M_U\models \psi([f_1]_U,...,[f_n]_U).$$We will also need here the Rudin-Keisler order, given $M$-ultrafilters $U,W$ over $X,Y$ respectively and a function $f:X\to Y\in M$, we say that $f$ is a \textit{Rudin-Keisler (RK) projection} of $U$ to $W$, if $f_*(U):=\{B\subseteq Y\mid f^{-1}[B]\in U\}=W$. If there is am RK-projection of $U$ to $W$ we denote this by $W\leq_{RK}U$. It is well-known that $f_*(U)=W$ if and only if $k_f([g]_W):=[g\circ f]_U$ is a well-defined elementary embedding $k_f:M_W\to M_U$ such that $j_U=k_f\circ j_W$. 
For more information regarding ultrapowers and the large cardinal notions which are used in this paper (weakly compact, measurable, strong, supercompact) see \cite{Jech2003}. 

    \section{Localization at large cardinals}\label{section: ZFC}
The study of the various localization characteristics at a general inaccessible cardinal was initially addressed by the authors of \cite{BrendleFreidmanMontoya} who checked that the straightforward modification of the localization forcing $\mathbb{LOC}$ defined in \cite[Section 3.1]{BarJu95} (see Definition~\ref{Def: Loc}) can be used to manipulate these cardinals exactly the same as in the $\omega$ case.
    In this section, we study $\mfb_\kappa(\in^*)$ and $\mfd_\kappa(\in^*)$ under the assumption that $\kappa$ is a large cardinal. The first results show that we reach trivialities again, opposite to Fact~\ref{fact: Successor}, once we go to a measurable:
    \begin{theorem}\label{thm: trivialities at a measurable}
        Suppose that $\kappa$ is a measurable, then $$\mathfrak{d}_\kappa(\in^*)=2^\kappa\text{ 
  and }\mathfrak{b}_{\kappa}(\in^*)=\kappa^+.$$
    \end{theorem}
        \begin{proof}
        Let $U$ be a normal ultrafilter on $\kappa$ and $\mathcal{A}\subseteq SL_\kappa$ be a localizing family, i.e. $\forall f:\kappa\to\kappa\exists \varphi\in\mathcal{A}, f\in^*\varphi$. 
        
        For each $\varphi\in\mathcal{A}$, let $X_{\varphi}=[\varphi]_U$. Since $\varphi$ is a $\kappa$-slalom, by \L os theorem and normality of the ultrafilter we get
        $$X_{\varphi}\in [j_U(\kappa)]^{|[\id]_U|}=[j_U(\kappa)]^{\kappa}.$$ Note that since $V_U$ is closed under $\kappa$-sequences, $V_U$ and $V$ agree of which sets have cardinality $\kappa$. In $V$, let $\lambda=|\{X_{\varphi}\mid \varphi\in\mathcal{A}\}|$, to finish, we will prove that $\lambda=2^\kappa$. 
        Since each ordinal $\alpha<j_U(\kappa)$ is of the form $j_U(f)(\kappa)$ for some $f:\kappa\to\kappa$, $\bigcup_{\varphi\in\mathcal{A}}X_{\varphi}=j_U(\kappa)$. Recall that $|j_U(\kappa)|^V=2^\kappa$ and since for every $\varphi\in\mathcal{A}$, $|X_{\varphi}|=\kappa$, we must have that $|\mathcal{A}|=2^\kappa$.

        For the second part, for each $\alpha<j_U(\kappa)$ let $f_\alpha:\kappa\to\kappa$ be such that $[f_\alpha]_U=\alpha$. Let\footnote{We remark here that the functions $f_\alpha$ can be taken to be the canonical functions. This fact will be used later in the paper as the canonical functions are highly definable.} $\mathcal{F}=\{f_\alpha\mid\alpha<\kappa^+\}$, we claim that $\mathcal{F}$ must be $\in^*$-unbounded and therefore witnessing that $\mathfrak{b}_{\kappa}(\in^*)=\kappa^+$. Indeed, let $\varphi\in\prod_{\alpha<\kappa}[\kappa]^{|h(\alpha)|}$, then $[\varphi]_U\in [j_U(\kappa)]^\kappa$ and therefore it cannot contain $\kappa^+$. So there if $\alpha<\kappa^+$ such that $[f_\alpha]_U=\alpha\notin [\varphi]_U$. This means that for a $U$-measure one set of $\nu$'s (and in particularly, for unboundedly many $\nu$'s) $f_\alpha(\nu)\notin \varphi(\alpha)$, namely $\neg (f_\alpha\in^*\varphi)$.
    \end{proof}
Let us provide some precise $\mathsf{ZFC}$-restrictions, where we only assume the inaccessibility of $\kappa$.

    \begin{theorem}
        Suppose that $\kappa$ is inaccessible, $\CI$ is an ideal on $\kappa$ and let $h:\kappa\to \kappa$. Then \begin{enumerate}
            \item $\mathfrak{d}_h(\in^\CI)\cdot |\prod_{\alpha<\kappa} h(\alpha)/\CI|= |\kappa^\kappa/\CI|$.
            \item $\mathfrak{b}_h(\in^\CI)=\min\{|\prod_{\alpha<\kappa}h(\alpha)^+/\CJ|\mid \CI\subseteq \CJ\}$
        \end{enumerate} 
    \end{theorem}
    \begin{proof}
    Clearly, $\prod_{\alpha<\kappa}h(\alpha)/\CI\subseteq \kappa^\kappa/\CI$. Also, given a localizing family $\mathcal{A}$ of $(\CI,h)$-slaloms, we may assume that the family consists of distinct $(\CI,h)$-slaloms  modulo $\CI$. Enumerate $\kappa^{<\kappa}$ as $\kappa$ (which is possible by inaccessibility). Then the family $\mathcal{A}$ can be identified with a $\kappa^\kappa/\CI$ family and therefore $\mathfrak{d}_h(\in^\CI)\leq |\kappa^\kappa/\CI|$. We conclude that $\mathfrak{d}_h(\in^I)\cdot |\prod_{\alpha<\kappa} h(\alpha)/\CI|\leq |\kappa^\kappa/\CI|$.
    
    For the other direction, fix $\mathcal{A}$ a localizing family of $(\CI,h)$-slaloms.  Each $f:\kappa\to\kappa$ is localized by some $\varphi\in\mathcal{A}$, so there is $g\in \prod_{\alpha<\kappa}h(\alpha)$ such that on a measure one set of $\alpha$'s $f(\alpha)$ is the $g(\alpha)$-th element in $\varphi(\alpha)$. Hence $(g,\varphi)$ determines $f$ up to an $\CI$-null set.

    For $(2)$, consider a representative family $\mathcal{A}$ for $\prod h(\alpha)^+/\CJ$. Then   $\mathcal{A}$ must be unbounded, otherwise, there is $\varphi$, such that $|\varphi(\alpha)|=|h(\alpha)|$. But then also $\varphi(\alpha)\cap h(\alpha)^+$ is a localizing slalom and since $|\varphi(\alpha)\cap h(\alpha)^+|<h(\alpha)^+$ we have $\sup(\varphi(\alpha)\cap h(\alpha)^+)+1=g(\alpha)<h(\alpha)^+$. Since $g\in \prod h(\alpha)^+/\CJ$, its supposed to be caught by $\varphi$ on a measure one set in $\CI$ and therefore by a measure one set in $\CJ$. But $g$ is also represented by a function in $\mathcal{A}$, and this leads to a contradiction. 
    Given any $\mathcal{A}$ with $|\mathcal{A}|<\min\{|\prod h(\alpha)^+/\CJ|\mid \CI\subseteq \CJ\}$. If $\mathcal{A}$ cannot be $(h,\CI)$-slalomed, then it must be that $$X:=\{\alpha<\kappa\mid |\{f(\alpha)\mid \alpha\in \mathcal{A}\}|\geq h(\alpha)^+\}\in \CI^+.$$ Let $\CI\subseteq \CJ$ be an ideal with $\kappa\setminus X\in \CJ$. For each $\alpha$, let $\pi_\alpha$ be the transitive collapse of $\{f(\alpha)\mid \alpha\in\mathcal{A}\}$ to some $\theta_\alpha$. Then, each $f\in\mathcal{A}$ can be uniquely identified with $g\in \prod_{\alpha<\kappa}\theta_\alpha$. 

    We claim that $|\prod\theta_\alpha/\CJ|\geq |\prod h(\alpha)^+/\CJ|$. Given $f\in \prod h(\alpha)^+$ map it to $[f^*]$, where $f^*(\alpha)=\min(f(\alpha),\theta_\alpha^+)$. Then for any $\alpha\in X$, $f^*(\alpha)=f(\alpha)$ so. It follows that 
    $|\mathcal{A}|=|\prod_{\alpha<\kappa}\theta_\alpha|\geq |\prod_{\alpha<\kappa}h(\alpha)^+/\CJ|$, contradiction.
    \end{proof}
    \begin{question}
        Is the previous theorem true for weakly inaccessible?
    \end{question}
Recall that if $\CI$ extends the bounded ideal then $|\kappa^\kappa/\CI|=2^\kappa$ and item $(1)$ above  translates to:
$$\mathfrak{d}_h(\in^\CI)\cdot |\prod_{\alpha<\kappa}h(\alpha)/\CI|=2^\kappa.$$
The reason is that we can code every $X\subseteq\kappa$ as a function $\alpha\mapsto X\cap \alpha$ which in turn can be coded a function from $\kappa$ to $\kappa$ (since $\kappa$ is strongly inaccessible) and each two distinct such functions are different modulo $\CJ^\kappa_{bd}$.

\begin{corollary}\ {} 
    \begin{enumerate}
        \item  If there is an ultrafilter $U \supseteq \CI^*$ such that\footnote{For a linear order $\mathbb{L}=(L,\leq_L)$, for any $\ell\in L$ we denote $(-\infty,\ell)=\{\ell'\in\mathbb{L}\mid \ell'<_L\ell\}.$} $|(-\infty,[h]_U)|<2^\kappa$, then $\mathfrak{d}_h(\in^\CI)=2^\kappa$.
        \item $\mathfrak{b}_h(\in^\CI)=\min\{|(-\infty,[h^+]_U)|\mid \CI^*\subseteq U\text{ is an ultrafilter}\}.$
    \end{enumerate} 
\end{corollary}
\begin{proof}
Both $(1)$ and $(2)$, follows from the fact that if $$\CI\subseteq \CJ\Rightarrow |\prod_{\alpha<\kappa}g(\alpha)/\CJ|\leq |\prod_{\alpha<\kappa}g(\alpha)/\CI|,$$ and the Prime Ideal Lemma: every ideal can be extended to a prime ideal.
\end{proof}
If $M_U$ is well-founded we can just consider the $V$-cardinality of the ordinal $[h]_U$. This way, we can start deriving corollaries for large cardinals. First, at measurable cardinals we can recover and extend Theorem~\ref{thm: trivialities at a measurable}:
\begin{corollary}\label{cor: trivialities at a measurable}
    If $\CI^*$ can be extended to a normal  ultrafilter then $\mathfrak{d}_\kappa(\in^\CI)=2^\kappa$ and $\mathfrak{b}_\kappa(\in^\CI)=\kappa^+$. Moreover, $\mathfrak{d}_{\id^+}(\in^\CI)=2^\kappa$.
\end{corollary}
\begin{proof}
    In this case, let $\CI^*\subseteq U$ be a normal measure $[\id]_U=\kappa<2^\kappa$ and $[\id^+]_U=(\kappa^+)^{M_U}=\kappa^+$. For the moreover part, if $\mathfrak{d}_{\id^+}(\in^\CI)<2^\kappa$, then $\kappa^+<2^\kappa$ as $\kappa^+ \leq \mfd_{\id^+}(\in^\CI)$. But then 
    $|[\id^+]_U|^V=\kappa^+<2^\kappa$, which implies $\mathfrak{d}_{\id^+}(\in^\CI)=2^\kappa$, contradiction.
\end{proof}
Note that the previous corollary applies to the bounded and non-stationary ideals on a measurable.
For the bounding number, we obtain further restrictions at large cardinals. 
    \begin{corollary}
    Let $\CI\subseteq U$ be a $\sigma$-complete ultrafilter on $\kappa$ and $h:\kappa\to\kappa$ be any function. Then  \begin{enumerate}
        \item $\mathfrak{b}_{h}(\in^\CI)\leq|([h]_U^+)^{M_U}|^V\leq [h]_U^+$.
        \item $\mathfrak{d}_h(\in^\CI)<2^\kappa$ implies $2^\kappa=|[h]_U|^V.$
    \end{enumerate}  
    
    \end{corollary}
 Hence if $\kappa$ is measurable then $\mathfrak{b}_{\id^+}(\in^*)\leq\kappa^{++}$, $\mathfrak{b}_{\id^{++}}(\in^*)\leq \kappa^{+3}$ and so on. But also if $\mathfrak{d}_{\id^{++}}(\in^*)<2^\kappa$ then $2^\kappa\leq \kappa^{++}$.
 In Section \ref{cons2}, we will see that it is consistent on a measurable cardinal that $\mathfrak{b}_{\id^+}(\in^*)=\kappa^{++}$ and that $\mathfrak{d}_{\id^{++}}(\kappa)=\kappa^+<2^\kappa$.
 \begin{corollary}
     There cannot be two functions $h_1,h_2$ and a $\sigma$-complete ultrafilter $U$ such that:
     \begin{enumerate}
         \item $|[h_1]_U|^V<|[h_2]_U|^V$.
         \item $\mathfrak{d}_{h_1}(\in^*),\mathfrak{d}_{h_2}(\in^*)<2^\kappa$.
     \end{enumerate}
 \end{corollary}
 \begin{remark}
     In \cite{VANDERVLUGT2025103582}, it was asked whether there could be functions $h_0(\alpha)<h_1(\alpha)<h_2(\alpha)=2^{h_0(\alpha)}$ such that $\mathfrak{d}_{h_2}(\in^*)<\mathfrak{d}_{h_1}(\in^*)<\mathfrak{d}_{h_0}(\in^*)$. The previous corollary gives some limitations for this situation to occur at a measurable cardinal: it must be that for any ultrafilter $U$, $\kappa^{+3}\leq 2^\kappa=|[h_1]_U|^V=|[h_2]_U^V|$, so  $h_1=\id^{++}$ for example is ruled out. 
 \end{remark}

Some of the arguments we gave above at a measurable cardinal work for smaller large cardinals. Let us recall the definition of \textit{completely ineffable} cardinals from \cite{AbramsonHarringtonKleinbergZwicker:FlippingProperties}. A set $\emptyset\neq R\subseteq P(\kappa)$ is called a \textit{stationary class} if every $A\in R$ is stationary and $R$ is upwards closed with respect to inclusion. 
\begin{definition}
    A cardinal $\kappa$ is \textit{completely ineffable}, if there is a stationary class $R$ such that for every $A\in R$ and $F:[A]^2\to 2$, there is $B\in R$ such that $F\restriction [B]^2$ is constant.
\end{definition}
Nielsen and Welch \cite{NIELSEN_WELCH_2019} proved that being completely ineffable is equivalent to having a winning strategy in the normal filter game of length $\omega$. This gives also the following characterization:
\begin{theorem}
    $\kappa$ is completely ineffable if and only if there is a set forcing generic extension $V[G]$ in which there is a weakly amenable normal $V$-ultrafilter.
\end{theorem}

\begin{corollary}\label{Cor: completely ineff}
    If $\kappa$ is completely ineffable then $\mathfrak{b}_\kappa(\in^*)=\kappa^+$.
\end{corollary}

\begin{proof}
Since $\kappa$ is weakly amenable, there is a set forcing extension $V[G]$ containing a normal, weakly amenable $V$-ultrapower, $U$. In $V[G]$, let $j:V\to M_U$ be the ultrapower. By normality and weak amenability, $M_U$ is well-founded up to $(2^\kappa)^{+M_U}$ and $(\kappa^+)^{M_U}=(\kappa^+)^{V}$,  (since $P(\kappa)^V=P(\kappa)^{M_U}$). Let us prove that the canonical functions are unbounded: suppose not, then there is $\varphi:\kappa\to P(\kappa)\in V$ which localizes each canonical function. Then $M_U\models \kappa^+\subseteq [\varphi]_U$ and also by normality $M_U\models |[\varphi]_U|=\kappa$, contradiction.
\end{proof}
\begin{question}
    Suppose $\kappa$ is completely ineffable. Is $\mathfrak{d}_\kappa(\in^*)=2^\kappa$? Do we get further restriction on $\mathfrak{b}_{h}(\in^*)$ for $h>\id$?
\end{question}
In fact we do not even know the following:
\begin{question}
    Is there a large cardinal notion strictly weaker than a measurable cardinal which implies that $\mathfrak{d}_\kappa(\in^*)=2^\kappa$?
\end{question}
    A natural question along these lines concerns what the possible behavior of the localization cardinals can be at a weakly compact. Using the forcing from the next section, we will see that it is consistent that $\mathfrak{b}_\kappa(\in^*)>\kappa^+$ for a weakly compact cardinal $\kappa$. In fact this will even hold at a strongly unfoldable cardinal.

\section{Localization forcing and large cardinals}\label{forcing} 
In this section we look at the main forcing for changing $\mfb_h(\in^*)$ and $\mfd_h(\in^*)$ for $h : \kappa \to \kappa$ and apply some observations about it to show that certain instances can preserve measurability. The basic case where $h= \id$ was initially introduced in \cite[Definition 50]{BrendleFreidmanMontoya}, itself the obvious generalization of the localization forcing on $\omega$, as presented in e.g. \cite[Section 3.1]{BarJu95}, see also \cite[Section 3.8]{degreesswitzer} for a more in-depth treatment of the properties of this forcing notion. Our presentation is simply a further generalization to allow for arbitrary width slaloms. We begin with the main definitions. 

\subsection{The Forcing Notion $\mathbb{LOC}_{h, \kappa}$}
For simplicity in this subsection fix an uncountable cardinal $\kappa$ so that $\kappa^{<\kappa} = \kappa$. Let us introduce the main forcing notion:
\begin{definition}\label{Def: Loc}
    Let $h:\kappa\to \kappa$ be a function,  $\mathbb{LOC}_{h,\kappa}$ consists of pairs $(\sigma,F)$ where $\sigma:\gamma\to P(\kappa)$ for some $\gamma<\kappa$ such that $|\sigma(\alpha)|\leq |h(\alpha)|$. $F:\kappa\to P(\kappa)$ is a function such that $|F(\alpha)|\leq|h(\gamma)|$ for every $\alpha<\kappa$. Define $(\sigma,F)\leq (\tau,G)$ if $\tau\subseteq \sigma$ and 
    $G(\alpha)\subseteq F(\alpha)$ and for every $\alpha\in dom(\sigma)\setminus dom(\tau)$, $G(\alpha)\subseteq \sigma(\alpha)$.
\end{definition}
For a condition $p = (\sigma, F)\in \mathbb{LOC}_{h, \kappa}$ as above we refer to $\sigma$ as the {\em stem} of $p$ and call any such $\sigma$ a {\em stem}. The intention of this forcing is to add, similar to Hechler forcing, a slalom that localizes every ground model function. Suppose that $G\subseteq \mathbb{LOC}_{h,\kappa}$ is $V$-generic, and let $\varphi_G:=\bigcup_{(\sigma,F)\in G}\sigma$.
An easy density argument shows the following.
\begin{proposition}
     $\varphi_G$ is an $h$-slalom and for every $f:\kappa\to\kappa\in V$, $f\in^* \varphi_G$.
\end{proposition}
\begin{remark}
    Note that what is perhaps a more natural definition here is to have conditions $(\sigma,F)$ as above but the cardinality restriction on $F$ to be $|F(\alpha)|\leq |h(\alpha)|$ rather than $|h(|\sigma|)|$. However the two forcing notions are equivalent via properly coding $[\kappa]^{h(\alpha)}$ as ordinals below $\kappa$, as done in the proof of proposition \ref{subseteqcharacterization}.
\end{remark}
Modulo the fact that $\mathbb{LOC}_{h, \kappa}$ is (nearly) ${<}\kappa$-closed and $\kappa^+$-c.c. (proved below) we therefore easily get the following.

\begin{proposition}
    Let $\kappa$ be inaccessible and $h:\kappa \to \kappa$ increasing. Let $\kappa < \lambda < \mu$ be cardinals with $\lambda$ regular and $\cf(\mu)>\kappa$. The following are consistent.

    \begin{enumerate}
        \item $\mfb_h(\in^*) = 2^\kappa = \lambda$
        \item $\mfb_h(\in^*) = \mfd_h(\in^*) = \lambda < \mu = 2^\kappa$
    \end{enumerate}
\end{proposition} 

\begin{proof}
    For $(1)$ simply iterate $\mathbb{LOC}_{h, \kappa}$ with ${<}\kappa$-supports for $\lambda$-many steps. For $(2)$ begin in a model of $2^\kappa = \mu$ and iterate for $\lambda$-many steps. See \cite[Proposition 52]{BrendleFreidmanMontoya} for more details.
\end{proof}

We now examine some of the properties of $\mathbb{LOC}_{h, \kappa}$. 
\begin{proposition}[{\cite[Lemma 5.1]{BrendleFreidmanMontoya}}]\label{basicfactsaboutLOC} Assume that $\kappa$ is inaccessible and that $h:\kappa \to \kappa$ is increasing. There is a dense set of conditions in $\mathbb{LOC}_{h,\kappa}$ for which the restriction to this set is $\kappa$-linked. 
\end{proposition}
\begin{remark}
First, to be clear, here, $\kappa$-{\em linked} means that there is a partition of the forcing notion into $\kappa$ many pieces, each of which has its elements pairwise compatible. Note however this is different being  $\kappa$-centered forcing notion i.e. the pieces of the partition are such that any ${<}\kappa$-many have a joint lower bound. The difference can be seen by the fact below.
\end{remark}

\begin{fact}[{\cite[Lemma 59]{BrendleFreidmanMontoya}}]
If $\kappa$ is a strongly inaccessible cardinal, $h : \kappa \to \kappa$ is monotone increasing and $\P$ is a $\kappa$-centered forcing notion, then for every $\P$-name of an $h$-slalom, $\dot\varphi$, there are $\kappa$-many $h$-slaloms, $\{\varphi_\alpha \mid \alpha \in \kappa\}\in V$ so that if $g:\kappa \to \kappa$ is not localized by any $\varphi_\alpha$ then $\forces_\P \check{g} \notin^* \dot\varphi$. In particular, no $\kappa$-centered $\P$ can add an $h$-slalom localizing $(\kk)^V$.

\end{fact}

Consequently $\mathbb{LOC}_{h, \kappa}$ can never have a $\kappa$-centered dense subset.

Next, regarding closure properties of $\mathbb{LOC}_{h,\kappa}$ , in \cite{BrendleFreidmanMontoya} it was inaccurately stated that $\mathbb{LOC}_{h, \kappa}$ is ${<}\kappa$-closed\footnote{The proof does not include details for the claim.}. 
\begin{example} \label{example: non closed}
    To see a counterexample, suppose $\sigma$ is a partial slalom with domain $\omega$, $\kappa > \omega_1$ and $h:\kappa \to \kappa$ is such that $h(\omega)$ is countable (say the identity). For each $\alpha < \omega_1$ let $F_\alpha$ be a function so that $F_\alpha(\omega) = \alpha$. Then clearly $(\sigma, F_\alpha)$ is a condition and if $\alpha < \beta$ then $(\sigma, F_\alpha) \geq (\sigma, F_\beta)$. However, there is no joint lower bound on these $\omega_1$-many conditions because any such $(\sigma', F)$ would have to have $\sigma' (\omega) \supseteq \omega_1$, which is not possible. This could be avoided by strengthening to a dense set, as we show below. 
\end{example}
Abstracting from the example above, observe that ${<}\kappa$-length sequences of conditions have lower bounds if the stems \textit{all properly extend} one another. Motivated by this, let us call this the {\em strict order}, denoted $\overline{\mathbb{LOC}}_{h, \kappa}$. Concretely $\overline{\mathbb{LOC}}_{h, \kappa}$ is the forcing notion defined to have the same underlining set as $\mathbb{LOC}_{h, \kappa}$ but with $p \leq_{strict} q$ if and only if $p\leq q$ with the normal order and $\sigma^p \supsetneq \sigma^q$. By what we have seen, $\overline{\mathbb{LOC}}_{h, \kappa}$ is ${<}\kappa$-closed and (has a dense subset which is) $\kappa$-linked. It is then not hard to check that these forcing notions are forcing equivalent.
\begin{proposition}
    The identity map is a dense embedding\footnote{See \cite[Def. 7.7]{Kunen}.} of $\overline{\mathbb{LOC}}_{h, \kappa}$ into $\mathbb{LOC}_{h, \kappa}$.
\end{proposition}



This will be useful in that we can treat $\mathbb{LOC}_{h, \kappa}$ therefore as a ${<}\kappa$-closed forcing notion, thus simplifying some of the arguments below. As such, we will implicitly work with the strict order moving forward.

Let us note that even with the strict order, Example \ref{example: non closed} shows that the forcing is not ${<}\kappa$-directed closed. It is impossible to adjust the definition of $\mathbb{LOC}_{h, \kappa}$ as we did with the strict order to make this forcing notion ${<}\kappa$-directed closed, since there are $\mathsf{ZFC}$ constraints on measurable cardinals which ensure that certain iterations of $\mathbb{LOC}_{h, \kappa}$ will not preserve measurability in general. The exact amount of directness is determined by $h$. The following proposition plays a crucial in showing the optimality of our results from the previous section:
\begin{proposition}\label{Prop: Semi-directedness of the one step}
    Suppose that $A=\{(\sigma_i,F_i)\mid i<h(\lambda)\}$ is a directed system of conditions such that $dom(\sigma_i)<\lambda$ and $\bigcup_{i \in h(\lambda)} \sigma_i$ has domain at least $\lambda$. Then $A$ has a lower bound.
\end{proposition}
\begin{proof}
    Let $\sigma=\bigcup_{i\in h(\lambda)}\sigma_i$ and $F(\alpha)=\bigcup_{i\in h(\lambda)} F_i(\alpha)$. Note that for each $i < h(\lambda)$ we have that for every $\alpha \geq \lambda$ $|F_i (\alpha)| \leq h(|\sigma_i|) < h(\lambda)$ and hence in particular $|F(\alpha)|\leq |h(\alpha)|$, hence $(\sigma,F)$ is a legitimate condition. Given $i< h(\lambda)$ we clearly have $\sigma_i\subseteq\sigma$, and $F_i(\alpha)\subseteq F(\alpha)$. For $\alpha\in \lambda\setminus {\rm dom}(\sigma_i)$ let $\beta$ be such that $\alpha\in {\rm dom}(\sigma_\beta)$. Since $\sigma_\beta,\sigma_i$ are compatible, we will have that $F_i(\alpha)\subseteq \sigma_\beta(\alpha)=\sigma(\beta)$.
\end{proof}

We will need a similar result for the ${<}\kappa$-support iteration of $\mathbb{LOC}_{h,\kappa}$. Concretely, this iteration is the following. For an ordinal $\delta$ we define $\mathbb{L}^h_\delta$ to be the set of all partial functions $p$ with domain a subset of $\delta$ of size ${<}\kappa$ so that for all $\alpha \in {\rm dom}(p)$ recursively we can define $p(\alpha)$ to be a $\mathbb L^h_\alpha$ name for a condition in $\mathbb{LOC}_{h,\kappa}$. Moving forward, we fix this notation to refer to such iterations. The following fact is obvious in the $\omega$ case but requires a bit of work in the uncountable case. 

\begin{lemma}
  Let $\delta$ be an ordinal. There is a dense set of conditions $p$ in the ${<}\kappa$-support iteration of $\mathbb{LOC}_{h,\kappa}$ so that for all $\alpha \in {\rm dom}(p)$ we have that  $p(\alpha)$ is of the form $p(\alpha)=(\check{\sigma}^r_i,\dot{F})$ for some $\sigma \in ([\kappa]^{<\kappa})^{{<}\kappa}$. \label{decide check names of stems}
\end{lemma}

\begin{proof}
     By replacing the standard order by the strict order we obtain that $\mathbb L_\delta^h$ is forcing equivalent to a ${<}\kappa$-closed forcing notion. Similarly for each $\alpha \leq \delta$ let $\dot{G}_\alpha$ denote the canonical name for the $\mathbb L^h_\alpha$ generic and, in the extension let this object be called $G_\alpha$. Note that for all $\alpha < \beta \leq \delta$ we have that $G_\alpha = G_\beta \cap \P_\alpha$. We make similar notational choices regarding things like $\mathbb L^h_{\alpha, \beta}$, the tail of the iteration. Note that by the definability of the forcing notion, for any $\alpha < \delta$ we have that $\mathbb L^h_\alpha$ forces the tail to be simply the ${<}\kappa$-supported iteration of $\mathbb{LOC}_{h, \kappa}$ indexed by $[\alpha, \delta)$ as computed in $V[G_\alpha]$. We will prove the lemma by induction on $\delta$. 
    
    \noindent \underline{Case 1}: $\delta = \xi + 1$ is a successor ordinal. Thus $\mathbb L^h_\delta = \mathbb L^h_\xi * \dot{\mathbb{LOC}}_{h, \kappa}$. 
    By the inductive assumption moreover we can pass to a dense set of $p \in \mathbb L^h_\alpha$ which has the form described in the statement of the lemma. Moreover, by closure every name for a stem is forced to be a ground model sequence and, by strengthening further we obtain the result in this case.

    \noindent \underline{Case 2}: $\delta$ is a limit ordinal. Let $p \in \mathbb L^h_\delta$. If ${\rm cf}(\delta) \geq \kappa$ then the support of $p$ is bounded in $\delta$ and we can apply our inductive hypothesis. Therefore we can assume without loss of generality that ${\rm cf}(\delta) < \kappa$ and ${\rm supp}(p)$ is cofinal in $\delta$. Let us fix that the cofinality of $\delta$ is some $\lambda < \kappa$ and choose a strictly increasing, cofinal sequence $\{\delta_i\mid i \in \lambda\} \subseteq \delta$. Applying our inductive hypothesis, we can recursively construct a sequence $\{p_i\mid i \in \lambda\}$ as follows:
    \begin{enumerate}
        \item $p = p_0$
        \item $p_{i+1} \leq p_i$ and every stem appearing in $p_{i+1} \hook \delta_i$ is a check name.
        \item If $\xi$ is a limit ordinal then $p_\xi$ is defined as follows. First, ${\rm supp}(p_\xi) = \bigcup_{i < \xi} {\rm supp}(p_i)$. Next, for each $\alpha \in {\rm supp}(p_\xi)$ define $p_\xi(\alpha)$ as the pair of names for the union of the stems of $p_i(\alpha)$ for $i < \xi$ and the name for the function mapping each $\zeta$ above the supremum of the the union of the stems to the union of the sets $\dot{F}_i^\alpha(\zeta)$ for $\dot{F}_i^\alpha$ the name for the second coordinate of $p_i(\alpha)$.
    \end{enumerate}

    Unwinding the third item note that $p_\xi(\alpha)$ is simply the name for lower bound on the $p_i(\alpha)$'s which exists by ${<}\kappa$-closure (of the strict order). Note moreover that, by applying item (2), for all $\i < \xi$ in this situation we have that $p_\xi \hook \delta_i$ has all of its stems decided as check names since the union of check names is literally the check name of the union. A straightforward verification now shows that $p_\lambda$ (in the parlance above) is the desired condition.
    \end{proof}

From now on we will treat iterations like this as restricted to this dense subset. 

\begin{proposition}\label{Prop: Semi-directedness of the iteration}
    Let $\lambda<\kappa$ and $\delta$ be any ordinal and $A=\{p_i\mid i<h(\lambda)\}$ is a directed system of conditions in the dense set of conditions from Lemma \ref{decide check names of stems} of $\mathbb L^h_\delta$ such that for every $\alpha \in \bigcup_{i < h(\lambda)}{\rm supp}(p_i)$ it is forced that the supremum of the stems at coordinate $\alpha$ has length at least $\lambda$. Then $A$ has a lower bound.
\end{proposition}
\begin{proof}
   Due to the dense set we restricted to, we can argue exactly as in Lemma \ref{decide check names of stems} coordinate-wise. There is no issue with the iteration, as everything is decided from the stems. 
\end{proof}

\subsection{Consistency Results and Large Cardinal Preservation} \label{cons2}
Let us apply use the iteration of $\mathbb{LOC}_{\kappa,h}$ to show that Corollary~\ref{cor: trivialities at a measurable} is optimal in the sense that the bounds may fail when $h \neq \id$. The difficulty is of course to preserve large cardinal notions after iterating $\mathbb{LOC}_{h,\kappa}$.  First notice that certain small large cardinals are easily seen to be preserved by iterating $\mathbb{LOC}_\kappa$. For example, inaccessibles or Mahlo cardinals, and therefore it is possible to alter the values of $\mathfrak{b}_\kappa(\in^*),\mathfrak{d}_\kappa(\in^*)$ and preserve those properties.

Moving a bit higher, it is alreacy unclear if weakly compact cardinals is preserved by the localization forcing. However, there is an easy fix here by invoking the work of Hamkins and Johnstone from \cite{Hamkins2010-HAMISU}.
   Recall the following definition due to Villaveces \cite{Villaveces98}.    \begin{definition}
         An inaccessible cardinal $\kappa$ is strongly unfoldable if for every ordinal $\theta$  and every transitive set $M$ of size $\kappa$ with
$\kappa\in M$, $M\models \mathsf{ZFC}^-$ and $M^{<\kappa}\subseteq M$ there is a transitive set $N$ and an elementary
embedding $j : M \to N$ with critical point $\kappa$ such that $\theta\leq j(\kappa)$ and $V_\theta\subseteq N$.
     \end{definition}
Intuitively, strongly unfoldable cardinals are to strong cardinals what weakly compact cardinals are to measurable cardinals. Villaveces \cite{Villaveces98} showed that strongly unfoldable cardinals are totally indescribable and in particular, $\Pi^1_1$-indescribable (i.e. weakly compact).
\begin{theorem}[Hamkins-Johnstone, see \cite{Hamkins2010-HAMISU}] \label{hamkinsindestructibility}
    If $\kappa$ is strongly unfoldable there is a generic extension in which it is still strongly unfoldable and moreover remains so in any further ${<}\kappa$-closed, $\kappa^+$-preserving forcing extension.
\end{theorem}


     Applying an iteration of $\mathbb{LOC}_\kappa$ to the model above we get following immediately. 
     \begin{corollary}
         Relative to the consistency of a strongly unfoldable cardinal the following are consistent. \begin{enumerate}
             
         \item There is a weakly compact cardinal
         $\kappa$ so that $\mathfrak{b}_\kappa(\in^*)>\kappa$.
         \item There is a weakly compact cardinal
         $\kappa$ so that $\mathfrak{d}_\kappa(\in^*)<2^\kappa$.
         \end{enumerate}
     \end{corollary}
     
     Note that here it is important that we can force with the strict order, as $\kappa$-strategically closed forcing notions alone might kill the weak compactness of a cardinal. 
     \begin{question}
         {} \ 
         \begin{enumerate}
             \item In the gap between strongly unfoldable and completely ineffable cardinals, is  $\mathfrak{b}_\kappa(\in^*)=\kappa^+$ always true? e.g. ineffable cardinals or subtle cardinals?
             \item Can a weakly compact cardinal be the first place where $\mathfrak{b}_\kappa(\in^*)>\kappa^+$?
             \item What is the consistency strength of the statement ``$\kappa$ is weakly compact and $\mathfrak{b}_\kappa(\in^*)>\kappa^+$"?
         \end{enumerate}
     \end{question}
     Next, let us move to higher cardinals, in which we know that $\mathfrak{b}_\kappa(\in^*)$ and $\mathfrak{d}_\kappa(\in^*)$ cannot be altered. The question here is about the cardinals $\mathfrak{b}_h(\in^*)$ and $\mathfrak{d}_h(\in^*)$ where $h\neq\id$. Recall that we denote by $\id^{+\xi}$ the function mapping $\alpha$ to the $\xi^\text{th}$-cardinal past $\alpha$, i.e. $\alpha^{+\xi}$. Also $2^{\id}$ denotes $\alpha\mapsto 2^\alpha$. 
\begin{theorem}\label{thm: preserving measurabiblity}
    Let $\kappa$ be a supercompact cardinal and assume $GCH$. Then there is a forcing extension $V_0$, in which $GCH$ holds above $\kappa$, and in $V_0$, the supercompactness of $\kappa$ is indestructible under $\mathbb{L}^{\id^{+\xi}}_\lambda$, where $\xi$ is below the first inaccessible above $\kappa$ and $\lambda$ is any ordinal below $\kappa^{+\xi+1}$.
\end{theorem}
\begin{proof}
    Let $\kappa$ be a supercompact cardinal, $2^\kappa=\kappa^+$, and let $\ell$ be a Laver function  \cite{Laver1978MakingTS}.
 Consider the Easton support iteration $\l \mathbb{P}_\alpha,\dot Q_\beta\mid \alpha\leq\kappa,\beta<\kappa\r$, such that for every $\alpha<\kappa$, $\dot Q_\alpha$ is a name for the trivial forcing unless $\alpha$ is inaccessible and $\ell(\alpha)$ is a $\mathbb{P}_\alpha$-name for a $\alpha$-closed forcing notion, in which case, we set $\dot Q_\alpha=\ell(\alpha)$. Let $V_0=V[G_\kappa]$, where $G_\kappa$ is $V$-generic for $\mathbb{P}_\kappa$. By the Easton support, $|\mathbb{P}_\kappa|=
 \kappa$ and in $V_0$, $\mathrm{GCH}$ holds above $\kappa$. 
 
 Let us claim that in $V_0$, $\kappa$ is indestructible  under $\mathbb{L}^{\id^{+\xi}}_\lambda$. By Lemma~\ref{decide check names of stems} we may restrict the form of conditions in $\mathbb{L}^{\id^{+\xi}}_\lambda$ to be partial functions $p$ with $\dom(p)\subseteq\lambda$ of size ${<}\kappa$ (the support of $p$) and for every $\alpha\in \dom(p)$, $p(\alpha)=(\check{\sigma}^p_\alpha,\dot{F}^p_\alpha)$, such that for every $\nu<\kappa$, $\dot{F}^p_\alpha(\nu)$ is a $\mathbb{L}^{\id^{+\xi}}_\alpha$-nice name for a subset of $\kappa$ of cardinality at most $|\sigma^p_\alpha|^{+\xi}$. 
 
Let $g_\kappa$ be $V_0$-generic for $\mathbb{L}^{\id^{+\xi}}_\lambda$. In $V_0$, let  $\mathcal{U}$ be a normal fine $P_\kappa(\kappa^{+\xi})$-ultrafilter such that $j_{\mathcal{U}}(\ell)(\kappa)=\mathbb{L}^{\id^{+\xi}}_\lambda$. Since $M_{\mathcal{U}}$ is closed under $\kappa^{+\xi}$-sequences, elementarity implies that
$$j_{\mathcal{U}}(\mathbb{P}_\kappa*\mathbb{L}^{\id^{+\xi}}_\lambda)=\mathbb{P}_\kappa*\mathbb{R}_{\lambda}*\mathbb{P}_{(\kappa,j_{\mathcal{U}}(\kappa))}*j_{\mathcal{U}}(\mathbb{L}^{\id^{+\xi}}_\lambda)$$
We would like to lift the embedding $j_{\mathcal{U}}$ to $V[G_\kappa*g_\kappa]$. First, we note that $G_\kappa*g_\kappa$ is also $M_{\mathcal{U}}$-generic for $\mathbb{P}_\kappa*\mathbb{L}^{\id^{+\xi}}_\lambda$. By the $\kappa^+$-chain condition of $\mathbb{P}_\kappa*\mathbb{L}^{\id^{+\xi}}_\lambda$, $M_{\mathcal{U}}[G_\kappa*g_\kappa]$ is closed under $\kappa^{+\xi}$-sequences from $V[G_\kappa*g_\kappa]$ (see for example \cite{CummingsHand}). Now, we can construct $H\in V[G_\kappa*g_\kappa]$ which is $M_U[G_\kappa*g_\kappa]$-generic for $\mathbb{P}_{(\kappa,j_{\mathcal{U}}(\kappa))}$. This is possible since from the perspective of $V[G_\kappa*g_\kappa]$, and by the cardinal arithmetic assumption, there are only $\kappa^{+\xi+1}$-many dense sets to meet\footnote{If $D\in M_{\mathcal{U}}$ is dense for $j_{\mathcal{U}}(\mathbb{P}_\kappa)$, then $D=[f]_{\mathcal{U}}$, where $f:P_\kappa(\kappa^{+\xi})\to P(\mathbb{P}_\kappa)$. Since $\mathbb{P}_\kappa$ has size $\kappa$, we have that at most $(2^\kappa)^{\kappa^{+\xi}}=\kappa^{\xi+1}$-many such dense sets.}. However, the forcing $\mathbb{P}_{(\kappa,j_{\mathcal{U}}(\kappa))}$ is $\kappa^{+\xi+1}$-closed from the perspective of $V[G_\kappa*g_\kappa]$ as $\kappa^{+\xi}$ is below the next $M_{\mathcal{U}}$-inaccessible which makes the forcing $\mathbb{P}_{(\kappa,j_{\mathcal{U}}(\kappa))}$ closed (this is where the assumption that $\xi$ is below the next inaccessible is used).
Also note that by the Easton support, $j_{\mathcal{U}}''G_\kappa=G_\kappa$. Hence, by the usual Silver criterion, in $V[G_\kappa*g_\kappa]$, $j_{\mathcal{U}}$ lifts to $j:V[G_\kappa]\to M_{\mathcal{U}}[G_\kappa*g_\kappa*H]$.

Next, we lift $j$ to $V[G_\kappa*g_\kappa]$ for which we will need to construct a master condition. The obstacle here is that $j_{\mathcal{U}}(\mathbb{L}^{\id^{+\xi}}_\lambda)$ is not $\kappa$-directed closed. So we need to argue differently that there is a lower bound for $j''g_\kappa$. 

For this, we note that $j''g_\kappa=\{j(q)\mid q\in G\}$ is a collection of $|\lambda|$-many elements of $j_{\mathcal{U}}(\mathbb{L}^{\id^{+\xi}}_\lambda)$ in $M_{\mathcal{U}}[G_\kappa*g_\kappa*H]$, and for every $q\in G$, and $\alpha\in \dom(j(q))$, $\dom(\sigma^q_\alpha)\subseteq\kappa$. Since $|\lambda|\leq \kappa^{+\xi}=\id^{+\xi}(\kappa)$, we can apply  Proposition~\ref{Prop: Semi-directedness of the iteration}, and find a lower bound $p^*\in j_{\mathcal{U}}(\mathbb{L}^{\id^{+\xi}}_\lambda)$ for $j_{\mathcal{U}}''g_\kappa$. 

Finally, we can construct an $M_{\mathcal{U}}[G_\kappa*g_\kappa*H]$-generic filter $g_{j(\kappa)}$ for $j_{\mathcal{U}}(\mathbb{L}^{\id^{+\xi}}_\lambda)$, with $p^*\in g_{j(\kappa)}$, since again there are only $\kappa^{+\xi+1}$-many dense subsets to meet and the forcing is sufficiently closed.
It follows that in $V[G_\kappa*g_\kappa]$, $j_{\mathcal{U}}$ lifts, and therefore $\kappa$ remains supercompact.
\end{proof}
\begin{remark}
    It is also possible to preserve an inaccessible degree of supercompactness. However, to do so, we would have to force $V$-generic filters at for the localization forcing at those levels and transfer the upper part of the generic along the supercompactness embedding (see \cite{CummingsHand}). 
\end{remark}
\begin{corollary}\label{cor: consistency of b}
    For every $\eta$ and every $\kappa<\lambda=\cf(\lambda)\leq\kappa^{+\eta+1}$ it is consistent relative to an supercompact cardinal that $\mathfrak{b}_{\id^{+\eta+1}}(\in^*)=\lambda$ and $\kappa$ is measurable. In particular, for any $\eta>0$ it is consistent that $\kappa$ is measurable and $\mathfrak{b}_{\id^{+\eta+1}}(\in^*)>\mathfrak{b}_\kappa(\in^*)$.
\end{corollary}
\begin{proof}
    Applying Theorem~\ref{thm: preserving measurabiblity} for $\xi=\eta+1$, we obtain a model such that $\mathbb{L}^{\id^{+\eta}}_\lambda$ preserves the measurability of $\kappa$. For the second part we use Theorem \ref{thm: trivialities at a measurable}.
\end{proof}
The previous corollary answers Question 71 from \cite{BrendleFreidmanMontoya}, also posed as \cite[Question 4.3]{vandervlugtlocalizationcardinals}.
Recall that by Corollary~\ref{cor: trivialities at a measurable}, $\mathfrak{b}_{\id^{+\eta}}(\in^*)\leq \kappa^{+\eta+1}$, so in light of Corollary~\ref{cor: consistency of b} the only remaining case is whether we can get $\mfb_{\id^{+\eta}}(\in^*)=\kappa^{+\eta+1}$. A similar argument will now show the answer is positive. The main point here is to strengthen Theorem~\ref{thm: preserving measurabiblity}.
\begin{theorem}\label{thm: preserving measurabiblity2}
    Let $\kappa$ be a supercompact cardinal and assume $GCH$. Then there is a forcing extension $V_0$, in which $GCH$ holds above $\kappa$, and in $V_0$, the supercompactness of $\kappa$ is resilient to the ${<}\kappa$-support iteration of length $\kappa^{+\xi+1}$,  of $\mathbb{LOC}_{\id^{+\xi}, \kappa}$, where $\xi$ is any ordinal.
\end{theorem}
\begin{proof}
    We keep the notations from Theorem~\ref{thm: preserving measurabiblity}. The argument starts the same and the problem concentrates as expected in the construction of a master condition for $j_{\mathcal{U}}''g_\kappa$ in $j_{\mathcal{U}}(\mathbb{L}^{\id^{+\xi}}_{\kappa^{+\xi+1}})$. To do that, we need a refinement of Proposition \ref{Prop: Semi-directedness of the iteration}. Fixing $\alpha<\kappa^{+\xi+1}$, a stage of the iteration, by the $\mathrm{GCH}$ above $\kappa$, there are only $\kappa^{+\xi}$ nice $\mathbb{L}^{\id^{+\xi}}_{\alpha}$-names for a subset of $\kappa$ and only $\kappa^{+\xi}$-many functions $F:\kappa\to \{\mathbb{L}^{\id^{+\xi}}_{\alpha}\text{-nice names for subsets of }\kappa\}$. Consider
$j_U''g_\kappa$, and let us form a master condition $p^*\in j_{\mathcal{U}}(\mathbb{L}^{\id^{+\xi}}_{\kappa^{+\xi+1}})$. First let $\dom(p^*)=j_{\mathcal{U}}''\kappa^{+\xi+1}$. For each $\alpha<\kappa^{+\xi+1}$, we define $p^*(j_{\mathcal{U}}(\alpha))=(\sigma^*_\alpha,F^*_\alpha)$, where $\sigma^*_\alpha=\bigcup_{p\in g_\kappa}\sigma^p_\alpha$, and for each
$\nu<j_{\mathcal{U}}(\kappa)$, $F^*_\alpha(\nu)$ is defined to be a nice name for $\bigcup_{p\in g_\kappa}j_{\mathcal{U}}(F^p_\alpha)(\nu)$. It is important to note here that $F^*_\alpha(\nu)$ is forced to be a collection of cardinality $\kappa^{+\xi}=\id^{+\xi}(\kappa)$ for every $\alpha$, although $|g_\kappa|=\kappa^{+\xi+1}$.
Indeed, there are only $\kappa^{+\xi}$-many elements in the sets $\{F^p_\alpha\mid p\in g_\kappa\}$, and therefore only $\kappa^{+\xi}$-many elements in $\{j_{\mathcal{U}}(F^p_\alpha)\mid p\in g_\kappa\}$. 
\end{proof}
\begin{corollary}
    For any ordinal $\xi$ it is consistent relative to a supercompact cardinal, that $\kappa$ is measurable and $\mathfrak{b}_{\kappa,\id^{+\xi}}(\in^*)=\kappa^{+\xi+1}$. 
\end{corollary}
\begin{question}
    Is it consistent to have $\mathfrak{b}_{\id^+}(\in^*)<\mathfrak{b}_{\id^{++}}(\in^*)$ at any regular uncountable cardinal?
\end{question}
Regarding the dominating localization numbers we can use again Theorem~\ref{thm: preserving measurabiblity}:
\begin{corollary}\label{cor: consistency of d}
    Relative to a supercompact cardinal, for every $\xi$ and every $cf(\lambda)=\lambda\leq\mu\leq \kappa^{+\xi}$ with $\kappa^+\leq cf(\mu)$, it is consistent that $\mathfrak{d}_{\kappa,\id^{+\xi}}(\in^*)=\lambda$ and $2^\kappa=\mu$.
\end{corollary}
\begin{proof}
    From the model of Theorem~\ref{thm: preserving measurabiblity} force with  $\mathbb{L}^{\id^{+\xi}}_{\mu+\lambda}$.
\end{proof}
Again, note that this result is optimal by \ref{cor: trivialities at a measurable} since if $\mathfrak{d}_{\kappa^{+\xi}}(\in^*)<2^\kappa$, then $2^\kappa\leq\kappa^{+\xi}$. So the above corollary given the consistency of $\mfd_{\id^{++}}(\in^*)=\kappa^+<2^\kappa=\kappa^{++}$.

    \begin{remark}
        Regarding the consistency strength of the statement $\mathfrak{b}_{\id^+}(\in^*)>\kappa^+$, our result show that this is consistent from what seems to be an overkill-- a supercompact cardinal. We conjecture that the consistency strength is much lower, and in fact that it is the optimal one, i.e., a measurable cardinal $\kappa$ with $o(\kappa)=\kappa^{++}$. 
    \end{remark}
\begin{question}\label{Consistecy}
    What is the consistency strength of $\mfb_{\id^+}(\in^*)>\kappa$ at a measurable? What is the consistency of $\mfd_{\id^{++}}(\in^*)<2^\kappa$?
\end{question}

So far we have seen that changing the parameter $h$, we can play with the value of $\mathfrak{b}_h(\in^*)$.
To round up the picture, let us shopw that we can insist on having $h=\id$, altering the value of $\mathfrak{b}_\kappa(\in^\CI)$ and $\mathfrak{d}_\kappa(\in^\CI)$. Of course, the answer depends on the ideal $\CI$. For example, if $\CI$ can be extended to a normal ideal, then we know by Theorem~\ref{thm: trivialities at a measurable} that the answer is no. 

Let us show that there are ideals $\CI$ for which the value can differ: Suppose that $\CI=\jbd{\kappa}[X^*]$, where $X^*=\{\alpha^{++}\mid\alpha<\kappa\}$ is the set of double successor cardinals. Namely, $\CI^*$ (the dual filter) is generated by sets of the form $X^*\setminus \xi$ for $\xi<\kappa$.
There is a similar forcing to $\mathbb{LOC}_\kappa$ that adds a localizing $(\kappa,\mathcal{I}^*)$-slalom:
\begin{definition}
    Let $\mathbb{LOC}^{\CI^*}_{\kappa}$ consist of conditions $p=(\sigma,F)$ such that $\dom(\sigma)=X^*\cap \gamma^p+1$ and $F:X^*\to P(\kappa)$ is such that for every every $\alpha\in X^*\setminus \gamma^p+1$, $|F(\alpha)|\leq (\gamma^p)^{++}$. The order is completely analogous to $\mathbb{LOC}_\kappa$.
\end{definition}
It is important here that $\CI$ is chosen concretely so we can force a set in $\mathcal{I}^*$ by initial segments. To that end, the choice of $X^*$ has some degree of freedom. Again, the forcing $\mathbb{LOC}_\kappa^{\CI^*}$ is $\kappa$-linked, has an equivalent suborder which is $\kappa$-closed, and adds a $(\kappa,\CI^*)$-slalom $\varphi$ such that for every $f\colon\kappa\to\kappa\in V$ with $\{\alpha<\kappa\mid f(\alpha)\in \varphi(\alpha)\}\in\CI^+$.
We can iterate this forcing as before, and the iteration produces a model where $\mathfrak{b}_\kappa(\in^{\CI^*})>\kappa^+$. 
\begin{theorem}
    Relative to a supercompact cardinal, it is consistent that $\mathfrak{b}_\kappa(\in^{\CI^*})>\kappa^+$.
\end{theorem}
\begin{proof}
    The proof again is similar to Theorem~\ref{thm: preserving measurabiblity}, we iterate the forcing, exploiting the fact that $\CI^*$ has a canonical definition and we keep the notations from \ref{thm: preserving measurabiblity}. All the proof goes through without any changes, except the part where we have to find a master condition above $j_{\mathcal{U}}''g_\kappa$. To do that, we can simply note that $\kappa,\kappa^+\notin j_U(X^*)$, and therefore there is no risk by taking unions of $\kappa^{++}$-many sets  at each coordinate in $ j_U(X^*)\setminus \kappa$, to form $j_{\mathcal{U}}(F^p_\alpha)$. 
 \end{proof}
     \section{Club Localization Forcing and the Separation of More Cardinal Invariants} \label{section: club localization}

    The $\mathsf{ZFC}$ results from \S\ref{section: ZFC} at measurable apply equally well for the non-stationary ideal at $\kappa$. In this section we consider a variant of $\mathbb{LOC}_{h, \kappa}$ for adding a slalom capturing every ground model element of $\kk$ on a club and use this to show that the bounds are optimal for club capturing as well.  Nevertheless, in contrast to what was shown in \cite{CUMMINGSShelah} for the standard bounding and dominating numbers, we also show that the localization numbers for club capturing are not equal to the mod bounded counterparts, even at a measurable cardinal. 

    \begin{definition}
    Let $h:\kappa\to \kappa$ be strictly increasing,  $\mathbb{LOC}^{cl}_{h,\kappa}$ consists of pairs $(\sigma,F)$ where $\sigma:c\to P(\kappa)$ for some closed, bounded $c \subseteq\kappa$ such that $|\sigma(\alpha)|\leq |h(\alpha)|$ for each $\alpha \in c$. $F:\kappa\to P(\kappa)$ is a function such that $|F(\alpha)|\leq|h({\rm max}(c))|$ for every $\alpha<\kappa$. Define $(\sigma,F)\leq (\tau,G)$ if $\sigma$ end-extends $\tau$ i.e. $\sigma = \tau \hook {\rm dom}(\tau)$ and $G(\alpha)\subseteq F(\alpha)$ and for every $\alpha\in dom(\sigma)\setminus dom(\tau)$, $G(\alpha)\subseteq \sigma(\alpha)$.
\end{definition}
Note that generically, the domain of the slalom we are building will be a new club subset of $\kappa$ (indeed, it will be generic for the standard forcing to add a club via bounded approximations). Again an easy density argument shows the following.
\begin{proposition}
    Suppose that $G\subseteq \mathbb{LOC}^{cl}_{h,\kappa}$ is $V$-generic, and let $\varphi=\bigcup_{(\sigma,F)\in G}\sigma$ then $\varphi$ is a $(cl,h)$-slalom and for every $f:\kappa\to\kappa\in V$, $f\in^{cl} \varphi$.
\end{proposition}

Similarly we have the following. \begin{proposition}
    Let $\kappa$ be inaccessible and $h:\kappa \to \kappa$ monotone increasing. Let $\kappa < \lambda < \mu$ be cardinals with $\lambda$ regular and $\mu$ of cofinality ${>}\kappa$. The following are consistent. 

    \begin{enumerate}
        \item $\mfb_h(\in^{cl}) = 2^\kappa = \lambda$
        \item $\mfd_h(\in^{cl}) = \lambda < \mu = 2^\kappa$
    \end{enumerate}
\end{proposition} 

We now again examine some of the key properties of $\mathbb{LOC}^{cl}_{h, \kappa}$, emphasizing the differences with $\mathbb{LOC}_{h, \kappa}$. 

\begin{proposition} \label{basic facts about club localization forcing} Assume that $\kappa$ is inaccessible and that $h:\kappa \to \kappa$ is monotone increasing. 
\begin{enumerate}
    \item $\mathbb{LOC}^{cl}_{h,\kappa}$ is $\kappa$-closed.
    \item $\mathbb{LOC}^{cl}_{h,\kappa}$ is $\kappa$-centered. 
\end{enumerate}
\end{proposition}

\begin{proof}
    We remark that this is different than the case of $\mathbb{LOC}_{h, \kappa}$ as we obtain a genuine ${<}\kappa$-closed poset, with no need to pass to the strict order. Let $\langle p_i \mid i \in \gamma < \kappa\rangle$ be a decreasing sequence of conditions of some length $\gamma < \kappa$. If the stems of the conditions properly extend cofinally often then we can obtain a lower bound in almost the same way described in the proof of Proposition \ref{basicfactsaboutLOC}. Concretely, if $\lambda < \kappa$ and $\{(\sigma_i, F_i)\mid i \in \lambda\}$ is a decreasing sequence of conditions with strictly increasing stems then as before we get that $|\bigcup_{i \in \lambda} F_i({\rm sup}(\sigma_i))| \leq h({\rm sup}_{i \in \lambda} |\sigma_i|)$. This allows us to define a lower bound. The difference is now because the domain has to be closed we need to extend the union of the stems to one more point. This is now easy because we can append $\bigcup_{i \in \lambda} F_i({\rm sup}(\sigma_i))$ to the union of the stems. 
    
    If the stems stabilize at some initial stage, then we can simply find a $\sigma \supseteq \bigcup_{i < \gamma} \sigma_i$, where $\sigma_i$ is the stem of $p_i$ so that ${\rm dom} (\sigma) \setminus {\rm dom} \bigcup_{i < \gamma} \sigma_i$ consists of a single point $\xi$ larger than $\gamma$ and in particular, large enough that $\sigma  (\xi) \supseteq \bigcup_{i \in \gamma} F_i(\xi)$ for $F_i$ the second coordinate of $p_i$.  We remark for later purposes that this $\sigma$ depends not only on the stems but also on the functions $F_i$. 
    
    For $(2)$, since $\kappa = \kappa^{<\kappa}$, it suffices to show that if $\lambda < \kappa$ and $\{p_i\}_{i \in \lambda}$ all have the same stem then they have a lower bound. This is exactly however as in the first paragraph as we can omit ordinals from the domain to find a lower bound whose next point is large accommodate the promises. 
\end{proof}

A consequence of this fact, using \cite[Lemma 59]{BrendleFreidmanMontoya} is that the forcing $\mathbb{LOC}_{h, \kappa}^{cl}$ does not add an $h'$-slalom capturing the ground model elements of $\kk$ for any $h' \in \kk \cap V$ which is monotone increasing. This is also true for iterations, assuming they are of length at most $(2^\kappa)^+$, a fact whose proof however we delay until later. 

We also note for later that $\mathbb{LOC}^{cl}_{h, \kappa}$ is not directed closed, as can be seen by an example similar to that of Example \ref{example: non closed}. 
\begin{example}
    Partition $\omega_1$ (for example) into $\omega$ many pieces, say $\{A_n\; | \; n < \omega\}$. Let for each $n < \omega$ and each $\alpha \in A_n$ the condition $p_{n, \alpha} = (\sigma_{n, \alpha}, F_{n, \alpha})$ be defined as follows. First, the domain of $\sigma_{n, \alpha}$ is $n$ and for each $k < n$ we have $\sigma_{n, \alpha}(k) = k$. Let $F_{n, \alpha}(\omega) = \{\alpha\}$ and be the empty set otherwise. Then all the $p_{n, \alpha}$'s are compatible (and there is only $\omega_1$-many of them) but any joint lower bound $(\sigma, F)$ would have to have $F(\omega) = \omega_1$ and $\omega \in {\rm dom}(\sigma)$, which is not possible. 
\end{example}

However it has the same ``almost"-directed closure. The proofs of these are nearly identical to the corresponding facts for $\mathbb{LOC}_{h, \kappa}$ however the same difference appears again.

\begin{proposition} \label{semidirectedness of club version one step}
    Suppose that $A=\{(\sigma_i,F_i)\mid i<h(\lambda)\}$ is a directed system of conditions in $\mathbb{LOC}_{h, \kappa}^{cl}$ such that $\bigcup_{i \in h(\lambda)}{\rm dom}(\sigma_i)\geq\lambda$ or there an $i < h(\lambda)$ so that $\sigma_i \supseteq \sigma_j$ for all $j < h(\lambda)$. Then $A$ has a lower bound.
\end{proposition}

\begin{proof}
The proof is nearly verbatim as that of Proposition \ref{Prop: Semi-directedness of the one step} except for the case that $\sigma_i \supseteq \sigma_j$ for all $j < h(\lambda)$. Here, however the exact same argument for centeredness works.

\end{proof}

 Let us now explore properties of iterations of $\mathbb{LOC}^{cl}_{h,\kappa}$.
   \begin{lemma}\label{decide checknames for club version}
        For each $\gamma < \kappa$, there is a dense set of conditions $p$ so that there is a $\beta > \gamma$ and if $\alpha \in {\rm dom}(p)$ then $p \hook \alpha$ forces that the stem of $p(\alpha)$ has length $\beta$ and its restriction to $\gamma$ is a check name.
    \end{lemma}

Note this is similar, but not the same as Lemma \ref{decide check names of stems}. The difference here is that conditions may know how long their stems are but they cannot, in general, decide the value at the maximal point. This issue does not come up for the standard $\mathbb{LOC}_{h, \kappa}$ forcing since conditions do not need to have stems with closed domains.  

    \begin{proof}
    Let $\mathbb L^{cl, h}_\delta$ be the $\delta$-length iteration of $\mathbb{LOC}^{cl}_{h, \kappa}$ with ${<}\kappa$-support and $\mathbb L^{cl, h}_{\alpha,\delta}$ the $\mathbb L^{cl, h}_\alpha$-name for $\mathbb L^{cl, h}_{\gamma}$ as computed in $V^{\mathbb L^{cl, h}_{\alpha}}$, where $\gamma$ is the unique ordinal such that $\alpha+\gamma=\delta$. 
 Fix ordinals $\delta$ and $\gamma$ and assume inductively that for all $\alpha < \beta < \delta$ we have that $\mathbb L^{cl, h}_\alpha$ forces that $\mathbb L_{\alpha, \beta}^{cl, h}$ has a dense subset consisting of conditions whose stems are check names up to $\gamma$ and are forced to all have the same length as in the statement of the lemma. We want to show that the same holds for $\alpha < \delta$. The case where $\delta$ is a successor ordinal follows directly from the induction hypothesis plus the fact that $\mathbb{L}^{cl}_{h, \kappa}$ is ${<}\kappa$-closed. Hence let us assume $\delta$ is a limit ordinal. Let $p \in \mathbb L^{cl, h}_\delta$. If ${\rm supp}(p)$ is bounded below $\delta$ then again the induction hypothesis suffices so we assume that the support of $p$ is cofinal in $\delta$. Note this implies moreover that ${\rm cf}(\delta) := \lambda < \kappa$. Let $\{\delta_i\mid i \in \lambda\} \subseteq \delta$ be cofinal and strictly increasing. 

By mimicking the proof of Lemma \ref{decide check names of stems} with slightly more care to make the stems the same size, we obtain a condition $q \leq p$, all of whose stems are forced to have some size $\beta > \gamma$ and for each $\xi \in {\rm dom}(q)$ we have that the stem of $q(\xi)\hook \beta$ is a check name (but maybe the top level is not). This $q$ is then as needed.

    \end{proof}
  Using this we have the analogue of Proposition \ref{Prop: Semi-directedness of the iteration}  
\begin{proposition}
    Suppose that $\delta$ is an ordinal and $A=\{p_i\mid i<h(\lambda)\}$ is a directed system of conditions in the dense set of conditions described in Lemma \ref{decide checknames for club version} of the length $\delta$, $<\kappa$-support iteration of $\mathbb{LOC}^{cl}_{h,\kappa}$ such that for every $\alpha\in \bigcup_{i < h(\lambda)} {\rm supp}(p_i)$ either one of the alternatives of Proposition \ref{semidirectedness of club version one step} then $A$ has a lower bound.
\end{proposition}
The proof is nearly identical to that of Proposition \ref{Prop: Semi-directedness of the iteration} so we omit it. 

\begin{definition}\label{clb}
    Let $\kappa$ be a regular cardinal and $\P$ be a forcing notion which is ${<}\kappa$-closed and $\kappa$-centered as witnessed by $\mathbb{P} = \bigcup_{i\in \kappa} P_i$. We say that $\P$ has {\em canonical lower bounds} (as witnessed by $f$) if there is a function $f:\kappa^{<\kappa} \to \kappa$ so that for every $
    \lambda < \kappa$ and every decreasing sequence of conditions $\{p_i\mid i \in \lambda\}$ (if $i < j$ then $p_i \geq p_j$) with $p_i \in P_{\xi_i}$ for all $i < \lambda$ we have that there is a lowerbound $p^*$ on the sequence in $P_{f(\langle \xi_i\mid i \in \lambda\rangle)}$.
\end{definition}

If $\P$ is ${<}\kappa$-closed and $\kappa$-centered with canonical lower bounds we will simply write that $\P$ has clb. Note that if $\P$ has clb then it is ${<}\kappa$-closed and $\kappa^+$-c.c. In particular, it will not collapse cardinals, a fact used implicitly and repeatedly in what follows. Most standard examples of $\kappa$-centered generalizations of forcing notions that are well known in set theory of the reals to the higher Baire spaces have clb. Interestingly, the next lemma shows that $\mathbb{LOC}^{cl}_{h, \kappa}$ does not have clb.

\begin{lemma}\label{lemma:clb do not add clubslalom}
    Suppose that $\P$ has clb and $h:\kappa \to \kappa$ is a function. If $\dot{\phi}$ is a $\P$-name for an $h$-slalom there is a function $g:\kappa \to \kappa\in V$ such that $\Vdash_{\mathbb{P}}\check{g}\not\in^{cl}\dot{\phi}$. 
\end{lemma}

\begin{proof}
    Let $f:\kappa^{<\kappa} \to \kappa$ witness that $\P$ has canonical lower bounds. Let $\P = \bigcup_{\gamma < \kappa} P_\gamma$ be the centered pieces. For each $\gamma < \kappa$ let $\phi_\gamma$ be the slalom defined by $\phi_\gamma (\beta) = \{\delta \; | \; \exists p \in P_\gamma \; p \forces \check{\delta} \in \dot{\phi}(\check{\beta})\}$. Since $\P$ is $\kappa$-centered this is an $h$-slalom. For each $\alpha < \kappa$ denote by $s(\alpha)$ the set of all functions of size at most $\alpha$ with domain and range subsets of $\alpha$. Let $g:\kappa \to \kappa$ be such that for each $\alpha < \kappa$ we have $g(\alpha) \notin \bigcup_{s \in s(\alpha)} \phi_{f(s)}(\alpha)$. Note that such a $g$ exists since $|s(\alpha)|\times h(\alpha) < \kappa$ since $\kappa$ is strongly inaccessible. 

    Now towards a contradiction suppose $p \in \P$ forces that $\dot{C}$ names a club, and for all $\beta \in \dot{C}$ we have $\check{g}(\beta) \in \dot{\phi}(\beta)$. Let $\theta >> \kappa$ be sufficiently large and let $M \prec H_\theta$ be a model so that $p, \P, \dot{C}, g, \dot{\phi}, \kappa \in M$ and $|M| < \kappa$. Moreover, assume that if $\gamma := {\rm sup}(M \cap \kappa)$ then $M$ is closed under ${<}\gamma$-sequences and that $|M|=\gamma$. In particular, $M\cap \gamma = \gamma$. This is possible by inaccessiblity. By recursion using the ${<}\kappa$-closure define a sequence $\vec{p}_M :=\{p_i\; | \; i < \delta < \kappa\}$ so that $\vec{p}_M$ generates an $M$-generic filter. Note this is where we need the closure - as we need to assume that the initial segments of the sequence are in $M$. For each $i$ let $p_i \in P_{\gamma_i}$ and note that $\gamma_i \in M$ and hence $\gamma_i < \gamma$ with $p_0 = p$. Let $s = \langle \gamma_i\; | \;i \in \delta\rangle$ and note that $s \in s(\gamma)$. Thus by construction we have $g(\gamma) \notin \phi_{f(s)} (\gamma)$. Now let $q \in P_{f(s)}$ be the lower bound on the $p_i$'s given by the canonical lower bounds. We have by density plus the fact that $\gamma \cap M$ is cofinal in $\gamma$ that $q \forces \gamma \in \dot{C}$ and hence $q \forces \check{g}(\check{\gamma}) \in \dot{\phi}(\check{\gamma})$, but this contradicts the sentence a few lines ago. 
\end{proof}

Intuitively, the reason why clb fails for $\mathbb{LOC}_{h, \kappa}^{cl}$ is because the domain is required to be a closed set and hence if we take a sequence of conditions with strictly increasing stems we cannot simply put their union as the stem of the lower bound. Rather we need to extend to the closure of that union. Then, we are required to state explicitly which ordinals the slalom captures at the closure point. This cannot be decided from the stems alone - an observation we remarked upon in the proof of the $\kappa$-closure of the forcing. In the next section, we will put these facts together to obtain several inequalities.

\subsection{Consistency Results and Cardinal Characteristic Inequalities} \label{cons1}
We continue our discussion from the previous section on $\mathbb{LOC}^{cl}_{h, \kappa}$ and obtain several consistent cardinal characteristic inequalities. These results hold for an inaccessible cardinal $\kappa$ but the proofs are flexible enough that we will be able to apply them to preserve larger cardinals when interwoven with the lifting arguments from \S~\ref{cons2}. First, we separate the standard localization numbers from their club variants. The following lemma gives the requisite preservation result. Denote by $\text{Add}(\kappa,\lambda)$ the standard forcing to add $\lambda$ many Cohen subsets to $\kappa$.  

\begin{lemma} \label{clubversiondoesntaddfullslalom}
    Assume $\GCH$. Let $\kappa$ be strongly inaccessible, let $h, h'\in \kk$ be monotone increasing functions. Let $\lambda > \kappa$ be a regular cardinal. Suppose $\beta \leq \lambda^+$ and let $G_0 \subseteq \text{Add}(\kappa, \lambda)$. In $V[G_0]$, the ${<}\kappa$-support iteration of $\mathbb{LOC}^{cl}_{h, \kappa}$ of length $\beta$ does not add an $h'$-slalom capturing the Cohen functions in $\kk$ from $G_0$. 
\end{lemma}

\begin{proof}
    Let $\lambda$, $\beta$, $G_0$ etc be as in the statement of the lemma. Work in $V[G_0]$ and let $C= \{c_\xi \mid \xi \in \lambda\}$ denote the Cohen functions. Note that in $V[G_0]$ we have $2^\kappa = \lambda$ and hence $\beta \leq (2^\kappa)^+$. As no new elements of $\kk$ are added at the final stage, it suffices to consider the case where $\beta < (2^\kappa)^+$. By the Engelking-Karlowicz theorem, \cite{EKthm}, we can fix a family of $\kappa$ many functions $f_i$ for $i < \kappa$ whose domains are $\beta$ and whose ranges are the set of partial $h$-slaloms with closed domains (i.e. the stems of conditions in $\mathbb{LOC}^{cl}_{h, \kappa}$) so that every such function with domain of size ${<}\kappa$ is contained in one of the $f_i$'s. Note that if we force over this model now with a $\kappa$-closed partial order, the $f_i$'s will retain this property. 

    We want to show by induction on $\beta$ that the $\beta$-length ${<}\kappa$-supported iteration of $\mathbb{LOC}^{cl}_{h, \kappa}$, which recall we denote $\mathbb{L}^{cl, h}_{\beta}$, will preserve that the $c_\alpha$'s are unbounded with respect to $\in^*$ for $h'$-slaloms. Let $G_\beta$ be generic for this iteration and for each $\alpha < \beta$ let $G_\alpha$ be the generic up to $\alpha$. Inductively we will assume the following:

    \begin{center}
        For every $h'$-slalom $\phi$ in $V[G_0][G_\alpha]$ there are at most $\kappa$ many $\xi \in \lambda$ so that $c_\xi \in^* \phi$.
    \end{center}
    
    Clearly if we can show this for all $\beta$ this is enough. We also note that this is true of the Cohen functions in $V[G_0]$. Let us therefore assume for contradiction that there is a name $\dot\phi$ for an $h'$-slalom and the maximum condition of $\text{Add}(\kappa, \lambda) * \dot {\mathbb L}_{\beta}^{cl, h}$ forces that it will capture $\kappa^+$ many of the Cohen functions. There are four cases. 

    \noindent \underline{Case 1}: $\beta = \alpha + 1$ is a successor ordinal. Work in $V[G_0][G_\alpha]$. Since in this model the quotient of the iteration is just $\mathbb{LOC}_{h, \kappa}^{cl}$ we obtain that the quotient is $\kappa$-centered. Therefore in $V[G_0][G_\alpha]$ there are $\kappa$ many $h'$-slaloms $\l \phi_\alpha\mid\alpha \in \kappa\r$ so that if some $c_\xi$ avoids all of them then it will avoid $\dot \phi$. The inductive assumption then tells us that all but $\kappa$-many of the $c_\xi$'s will therefore be forced to avoid $\dot \phi$. 

    \noindent \underline{Case 2}: $\beta$ is a limit ordinal of cofinality greater than $\kappa$. In this case no new slaloms are added so our induction hypothesis suffices. 

    \noindent \underline{Case 3}: $\beta$ is a limit of cofinality $\kappa$. Let $\dot Z$ be the name for the set of $\xi < \lambda$ so that $c_\xi \in^*\dot \phi$ in the extension. By assumption this set is forced to have size at least $\kappa^+$. By pigeonhole, in the extension there will be some $\delta < \kappa$ so that $\kappa^+$ many of the $c_\xi$'s for $\xi \in \dot Z$ will be caught above $\delta$. As the supports are therefore bounded in $\beta$, we can find a single $\gamma < \beta$ so that for $\kappa^+$ many $\xi$ there is a $p_\xi$ whose support is contained in $\gamma$ and decides $\xi \in \dot{Z}$ with $c_\xi$ and $c_\xi$ is forced to be caught by $\phi$ above $\delta$. Work in $V[G_0][G_\gamma]$. In this model we can therefore evaluate $\kappa^+$-many elements of $\dot{Z}$ all of which are forced to be caught by $\dot \phi$ above $\delta$. Let $\bar{Z}$ be the set of these $\xi$. Now on the one hand, $\{c_\xi\mid \xi \in \bar{Z}\}$ cannot be caught by a single slalom in $V[G_0][G_\gamma]$ because of our inductive hypothesis. On the other hand, consider the function $\psi$ defined on $[\delta, \kappa)$ which maps $\alpha \mapsto \{c_\xi(\alpha)\mid \xi \in \bar{z}\}$. In $V[G_0][G_\beta]$ this function will extend to an $h'$-slalom but since we do not collapse cardinals, this means it must be such in this model, which is a contradiction. 

    \noindent \underline{Case 4}: $\beta$ is a limit ordinal and the cofinality of $\beta$ is less than $\kappa$. This is the hard case. As in Case 3, we can fix a set $\dot{Z}$ and an ordinal $\delta < \kappa$ so that in the extension $\dot{Z}$ will be a $\kappa^+$-sized set of $\xi$ so that $c_\xi \in^*\dot \phi$ and each $c_\xi$ is caught above $\delta$. Work momentarily in the extension and let $Z$ be the evaluation of $\dot{Z}$ here. 

    Recall that by Lemma \ref{decide checknames for club version} there is a dense set of conditions with a {\em height}: i.e. they force all their stems to have the same length, some fixed $\gamma$. An immediate pigeonhole argument shows the following claim.

    \begin{claim}
        There is a $\gamma < \kappa$ and an $i <\kappa$ so that $\kappa^+$-many $\xi$ are forced to be in $Z$ by a condition $p_\xi$ of height $\gamma$. Moreover, when evaluated in the extension, the map defined on the support of $p_\xi$ mapping $\alpha \mapsto {\rm stem}(p_\xi(\alpha))$ is extended by $f_i$.
    \end{claim}

    Let us say that a fondition $p$ is {\em compatible with} $f_i$ if it has the last property above. Back in $V[G_0]$, let $r$ force that some specific $i$ and $\gamma$ are as above. By strengthening if necessary we can moreover assume that $r$ forces all its stems to have some height $\theta > \gamma$ and each stem is a check name when restricted to $\gamma + 1$. Consider now the following set:

    \[Q = \{(p, \xi) \mid p || r \, \text{and} \, p \, \text{has height}\, \gamma \, \text{and} \, r \forces p \, \text{is compatible with}\, f_i \, \text{and}\,  p \forces \check \xi \in \dot Z\}\]

    The assumption on $r$ implies that $Q$ has size at least $\kappa^+$. Let $Q_0$ be the projection onto the first coordinate and $Q_1$ that onto the second coordinate. 

    \begin{claim}
        $Q_0 \cup \{r\}$ is ${<}\kappa$-directed i.e. every family of ${<}\kappa$ many conditions has a joint lower bound.
    \end{claim}

    \begin{proof}[Proof of Claim]
        Let $\{p_i\}_{i \in I}$ be some family of elements of $Q_0$ of size ${<}\kappa$. Note first that $r$ decides as check names all of the stems of the $p_i$'s, decides them to be the same as the stems of $r$ itself up to $\gamma$ and, since each $p_i$ is forced by $r$ to be compatible with $r$ we have that $r$ forces the stem of $r(\alpha)$ extends the stem of $p_i(\alpha)$ for each $i \in I$ and $\alpha \in {\rm dom}(p_i) \cap {\rm dom}(r)$. Moreover, if $\alpha \notin {\rm dom}(r)$, $r$ still forces the stems of the $p_i(\alpha)$'s to be compatible as they are both being forced to be compatible with $f_i$. Similarly $r$ forces that for each such $i$ and $\alpha$ that the stem of $r(\alpha)$ at $\zeta$ contains the ordinals given by the function in the second coordinate of $p_i(\alpha)$ at $\zeta$ for every $\zeta \in (\gamma, \theta]$. Putting these together, $r\hook \alpha$ forces that $r(\alpha)$ and $\{p_i(\alpha) \mid i \in I\}$ for a directed set of conditions exactly as in the proof of $\kappa$-centeredness of $\mathbb{LOC}^{cl}_{h, \kappa}$. This suffices therefore to find the lower bound.
    \end{proof}

    Now, consider the function $\psi$ defined on $(\delta, \kappa)$ so that $\alpha \mapsto \{c_\xi(\alpha)\mid \xi \in Q_1\}$. As in Case 3 our inductive hypothesis tells us that this function is not an $h'$-slalom (even on a tail). In particular, there is a $\gamma' > \gamma$ so that $\{c_\xi(\gamma ')\mid \xi \in Q_1\}$ has size greater than $h'(\gamma')$. Now consider some set of $h'(\gamma')^+$ many $\xi \in Q_1$ witnessing this and apply the Claim above to find an $r' \leq r$ so that $r'$ is a lower bound on all $p$ so that $\xi$ is in this set for $(p, \xi) \in Q$. This lowerbound now witnesses the same contradiction as in Case 3 so we are done. 
    
\end{proof}

As an immediate consequence, we have the following.

\begin{theorem}\label{separateclubsfrombounded}
    Let $h, h' \in \kk$ be strictly increasing. The following are consistent.
    \begin{enumerate}
        \item $\mfb_h(\in^*) = \kappa^+ < \mfb_{h'} (\in^{cl}) = \kappa^{++}$
        \item $\mfd_{h'}(\in^{cl}) = \kappa^+ < \mfd _h(\in^*) = \kappa^{++}$
    \end{enumerate}
\end{theorem}

\begin{remark}
    Again, part of the interest in this result is that it contrasts with the case of $\mfb$ and $\mfd$ which are provably equal to their club versions on an inaccessible (and less), see \cite{CUMMINGSShelah}. 
\end{remark}

\begin{proof}
    For (1), simply begin in a model of $\GCH$, add $\kappa^+$-many Cohen functions and then iterate with $\mathbb{LOC}_{h', \kappa}^{cl}$. By Lemma \ref{clubversiondoesntaddfullslalom} this suffices. 
    
    For (2), again assume $\GCH$ and first add $\kappa^{++}$-many Cohen subsets to $\kappa$. Enumerate these as $\{c_\alpha \mid \alpha \in \kappa^{++}\}$ and work in $V[c_\alpha\mid \alpha \in \kappa^{++}]$. In this model we will have that $\mfd_{h'}(\in^{cl}) =  \mfd _h(\in^*) = \kappa^{++} = 2^\kappa$. Now perform a $\kappa^+$-length iteration of $\mathbb{LOC}_{h', \kappa}^{cl}$. 
\end{proof}

The above proofs do not interfere with our preservation theorems for large cardinals. In particular, by apply the proof strategy above when $\kappa$ is supercompact we obtain the following two results.

\begin{theorem}
      Relative to a supercompact cardinal, for every $\xi>1$ the following are each consistent. 
      \begin{enumerate}
      \item
      There is a measurable cardinal $\kappa$ and $\mathfrak{b}_{\kappa,\id^{+\xi}}(\in^{cl})=\kappa^{+\xi + 1}>\kappa^+$. In particular, $\mathfrak{b}_{\id^{+\xi}}(\in^{cl})>\mathfrak{b}_\kappa(\in^{cl})$.
      \item 
    There is a measurable cardinal $\kappa$ and $\mathfrak{d}_{\kappa,\id^{+\xi + 1}}(\in^{cl})=\kappa^+ < \kappa^{+\xi+1} = 2^\kappa$. In particular, $\mathfrak{d}_{\id^{+\xi + 1}}(\in^{cl})<\mathfrak{d}_\kappa(\in^{cl})$.
      \end{enumerate}
\end{theorem}

\begin{theorem}
       For every $\xi_0, \xi_1 > 1$, relative to a supercompact cardinal $\kappa$ the following are consistent.
    \begin{enumerate}
        \item $\kappa$ is supercompact and $\mfb_{\id^{+\xi_0}}(\in^*) = \kappa^+ < \mfb_{\id^{+\xi_1}} (\in^{cl}) = \kappa^{++}$
        \item $\kappa$ is supercompact, $\kappa^+\leq cf(\lambda)\leq\lambda\leq\kappa^{+\xi}$, and $\mathfrak{d}_{\kappa,\id^{+\xi_1+1}}(\in^{cl})=\lambda$ and $2^\kappa=\kappa^{+\xi_0+1} = \mfd_{\kappa, \id^{+\xi_0 + 1}}(\in^*)$.
    \end{enumerate}
\end{theorem}

Next, we look to separate more cardinal characteristics. We mentioned in the preliminaries section, $\mfb(\in^*) \leq \mfb (\in^{cl}) \leq \mfb (\mathsf{p} \in^*)$ , and dually $\mfd(\in^*) \geq \mfd (\in^{cl}) \geq \mfd (\mathsf{p} \in^*)$. Having separated the first two of these we now separate the next two. Before doing so we recall the definitions of two more standard cardinal characteristics on $\kk$. See, e.g. \cite{FischerGaps, BROOKETAYLOR201737} for more.
\begin{definition}
    Let $\kappa$ be a regular cardinal. 
    \begin{enumerate}
        \item A family $\mathcal F \subseteq [\kappa]^\kappa$ is said to have the {\em strong intersection property}, denoted SIP, if every $\mathcal F' \subseteq \mathcal F$ of size ${<}\kappa$ has intersection size $\kappa$. The cardinal $\mathfrak{p}_\kappa$ is the least size of a family $\mathcal F$ with the SIP and no pseudointersection i.e. no $A\in [\kappa]^\kappa$ so that for all $B \in \mathcal F$ we have $A \setminus B$ has size ${<}\kappa$.

        \item A family $\mathcal S \subseteq [\kappa]^\kappa$ is called splitting if for every $A \in[\kappa]^\kappa$ there is a $B \in \mathcal S$ so that $|A \cap B| = |A \setminus B| = \kappa$. The {\em splitting number} $\mathfrak{s}_\kappa$ is the least size of a splitting family.
    \end{enumerate}
\end{definition}
Note that by aforementioned results of Suzuki \cite{suzuki} if $\mathfrak{s}_\kappa > \kappa$ then $\kappa$ is weakly compact. It follows that if $\kappa$ is not a large cardinal then it may not be true that $\mathfrak{p}_\kappa \leq \mathfrak{s}_\kappa$ (in the $\omega$ case this holds). However we have the following.

\begin{proposition}\label{Prop: relation between p and s on large}
        If $\kappa$ is measurable then $\mathfrak{p}_\kappa \leq \mathfrak{s}_\kappa$. 
    \end{proposition}

    \begin{proof}
        Let $\mu < \mathfrak{p}_\kappa$ and let $\{A_i\mid i \in \mu\} \subseteq [\kappa]^\kappa$. Let $\mathcal U$ be a measure on $\kappa$ and for each $i < \mu$ let $X_i \in \mathcal U$ be such that either $X_i \cap A_i = \emptyset$ or $X_i \subseteq A_i$. Applying $\mathfrak{p}_\kappa > \mu$ we can find an $X^* \subseteq^* X_i$ for all $i < \mu$ and hence $X^*$ witnesses that $\{A_i\mid i \in \mu\}$ is not a splitting family.
    \end{proof}

The following will allow us to separate $\mfb_{h}(\in^{cl})$ from $\mathfrak{p}_\kappa$. We refer the reader to \cite[Definition 4.4]{FischerGaps} for the definition of the higher Mathias forcing, $\mathbb M(\mathcal F)$. 

\begin{lemma}\label{lemma:preserve small b club}
    Let $\beta \leq (2^\kappa)^+$ be an ordinal. Inductively define a ${<}\kappa$-supported iteration $\{\P_\alpha,  \dot{\Q}_\alpha \mid \alpha < \beta\}$ by allowing $\dot{\mathcal F}_\alpha$ to be a $\P_\alpha$-name for a ${<}\kappa$-closed filter on $\kappa$ and $\forces_\alpha \dot{\Q}_\alpha = \mathbb M(\dot{\mathcal F}_\alpha)$. Then $\forces_\beta$``For every $h \in \kk \cap V$ which is monotone increasing there is no $h$-slalom $\varphi$ so that $f \in^{cl}\varphi$ for every $f \in \kk \cap V$". In particular, it is consistent that $\mathfrak{p}_\kappa = \mu > \mfb_h (\in^{cl}) = \kappa^+$ is consistent for any regular $\mu > \kappa^+$.
\end{lemma}

\begin{proof}
    Let $\beta$, $\{\dot{\mathcal F}_\alpha\}_{\alpha < \beta}$ and $h$ be as in the statement of the lemma. By \cite[Lemma 55]{BrendleFreidmanMontoya} we can pass to a dense set and hence assume without loss that if $p \in \P_\beta$ then for all $\alpha \in {\rm supp}(p)$ there is an $s \in\kappa^{<\kappa}$ so that $p\hook \alpha \forces p(\alpha) = (\check{s}, \dot{A})$ for some $\dot{A}$. Let us restrict to this dense subset. Since no new slaloms are added at stage $(2^\kappa)^+$ we can assume that $\beta \leq 2^\kappa$. By \cite[Lemma 4.5]{FischerGaps} $\mathbb M(\mathcal F)$ has clb for any $\kappa$-complete $\mathcal F$ and hence by \cite[Lemma 55]{BrendleFreidmanMontoya} the iteration is $\kappa$-centered. In general, it is not clear that $\P_\beta$ has clb, but in this specific case it turns out to be true, a fact we show now. By the Engelking-Karlowicz theorem, \cite{EKthm}, there are functions $f_i:\beta \to \kappa^{<\kappa}$ so that every partial function from $\beta$ to $\kappa^{<\kappa}$ of size ${<}\kappa$ is contained in one of the $f_i$'s. Let $P_i$ be the ${<}\kappa$-centered set of conditions $p$ so that if $\xi \in {\rm supp}(p)$ then the stem of $p(\xi)$ is $f_i(\xi)$. 


    First, as a preliminary observation observe that if $f:\beta \to \kappa^{<\kappa}$ is any (partial) function and $p_0, p_1 \in \P_\beta$ are such that ${\rm supp}(p_0), {\rm supp}(p_1) \subseteq {\rm dom}(f)$ and for all $i < 2$ and for all $\xi \in {\rm supp}(p_i)$ we have that $f(\xi)$ is the stem of $p_i(\xi)$ then $p_0$ and $p_1$ are compatible. Note moreover this will be true for not just two conditions but any family of less than $\kappa$ many will be jointly compatible under these conditions. 
    
    Next, for any  sequence of $f_i$'s of length $\lambda<\kappa$, say $\vec{f} = \{f_{i_\xi} \mid \xi < \lambda\}$ define a function $\lim \vec{f}$ as the (partial) function from $\beta$ to $\kappa^{<\kappa}$ defined by $\lim \vec{f} (\alpha) = s \in \kappa^{<\kappa}$ if and only if for a tail of $\xi < \eta < \lambda$ we have that $f_{i_\xi}(\alpha) \subseteq f_{i_\eta}(\alpha)$ and $s$ is the union of the $f_{i_\xi}(\alpha)$'s. In other words take the coordinate wise union of the stems given by the $f_i$'s if this is defined (on a tail) and leave it undefined otherwise. Note that, by the previous paragraph, if $\{p_i\mid i < \delta < \kappa\}$ is some family of conditions of size ${<}\kappa$ all of whose supports are contained in the domain of $\lim(\vec{f})$ and whose stems agree with $\lim (\vec{f})$ then they have a joint lower bound. Let us say that a condition like this is {\em compatible} with $\lim (\vec{f})$. The point is the following.
    \begin{claim}
Let $\lambda < \kappa$ and let $\{p_i\mid i \in \lambda\}$ be a decreasing sequence of conditions so that for each $i$ we have that $p_i \in P_{{\xi_i}}$. If $\vec{f}$ denotes the sequence of functions $\l f_{\xi_i} \mid i < \lambda \r$, then there is a lower bound on the sequence $p^*$ compatible with $\lim (\vec{f})$.
    \end{claim}

    \begin{proof}
        Let $p^*$ be the greatest lower bound of the sequence. In other words, for each $\alpha \in \bigcup_{i < \lambda} {\rm supp}(p_i)$ we define $p^*(\alpha)$ to be $(\sigma^*_\alpha, \dot{A}^*_\alpha)$ where $\sigma^*_\alpha$ is the check name for the union of the stems in the $\alpha^{\rm th}$-coordinate of the $p_i$'s and $\dot{A}^*_\alpha$ is the name for the intersection of the second coordinates (which exists by ${<}\kappa$-closure of the functions). For each $\alpha \in {\rm supp}(p^*)$, note that there is a tail of $i < \lambda$ so that $\alpha \in {\rm supp}(p)$ and for this tail the set of $f_{\xi_i}(\alpha)$'s is increasing hence $\alpha \in {\rm dom}(\lim \vec f)$ and is equal to the stem of $p^*(\alpha)$. The rest is now clear. 
    \end{proof}
    
    Given any sequence $\vec{f}$ of the functions $f_i$ of length less than $\kappa$, let $Q_{\vec{f}}$ be the set of conditions compatible with $\lim \vec{f}$. Note that by choosing $\vec{f}$ to be the constant function for some fixed $f_i$ we recover $P_i$ and hence a dense subset of $\P_\beta$ is covered by the $Q_{\vec{f}}$'s. Note that since $\kappa^{<\kappa} = \kappa$ there are only $\kappa$ many sequences $\vec{f}$. But now the proof of the claim above gives that the map $\vec{f} \mapsto \lim \vec f$ gives canonical lowerbounds, which completes the proof by Lemma \ref{lemma:clb do not add clubslalom}.

\end{proof}

A consequence of this lemma is the following.
\begin{theorem}
    Assume $\GCH$. Let $\kappa$ be strongly inaccessible and $h:\kappa \to \kappa$ strictly increasing. There is a cofinality preserving forcing extension in which $\mfb_h (\in^{cl}) = \kappa^+ < \mathfrak{p}_\kappa = \kappa^{++}$.
\end{theorem}

\begin{proof}
    Simply iterate to force with all possible $\mathbb M (\mathcal F)$ as in the description of the lemma above. Clearly good enough bookkeeping will ensure that $\mathfrak{p}_\kappa = \kappa^{++}$ holds in the extension while $\mfb_h (\in^*) = \kappa^+$ will be witnessed by the ground model functions - this is the content of Lemma \ref{lemma:preserve small b club}.
\end{proof}
Since $\mathfrak{p}_\kappa \leq \mfb (\mathsf{p} \in^*)$ we have the following immediate corollary.
\begin{corollary}
    It is consistent that $\mfb_h(\in^{cl}) < \mfb (\mathsf{p} \in^*)$.
\end{corollary}

Regarding $\mathfrak{p}_\kappa$ it turns out the other inequality is also consistent. 

\begin{lemma} \label{cohen witnesse are preserved}
    If $\mu > \kappa^+$ is a regular cardinal and $\dot{\mathbb L^{h}_\mu}$ is the $\text{Add}(\kappa,\kappa^+)$-name for the ${<}\kappa$-supported iteration of $\mathbb{LOC}_{h, \kappa}$ (for any $h \in V$) of length $\mu$. Then $\text{Add}(\kappa,\kappa^+) *\dot{\mathbb L^h_{\mu}}$ forces that $\mathfrak{p}_\kappa = \kappa^+ < \mfb_h(\in^*) = \mu$.
\end{lemma}

This is similar to the proof of \cite[Theorem 4.8]{FischerGaps}.
\begin{proof}
      Let $\vec{c} = \{c_i\mid i \in \kappa^+\}$ denote the generic Cohen subsets added by $\text{Add}(\kappa,\kappa^+)$. We want to show that this forms a witness to $\mathfrak{p}_\kappa = \kappa^+$ after forcing with $\dot{\mathbb L^h_\mu}$. Applying Lemma \ref{decide check names of stems} to $\text{Add}(\kappa,\kappa^+) * \dot{\mathbb L}_\mu^h$ we can assume there is a dense set of conditions of $\text{Add}(\kappa,\kappa^+) * \dot{\mathbb L}_\mu^h$ of the form $(p, \bar{a}, \dot{F})$ where $\bar{a} \in V$ is the set of stems of the conditions in the support and $p \forces \dot{F}$ is the set of possible promises in the support. Moreover, by the $\kappa^+$-cc it's clear that if $\dot{x}$ is a $\text{Add}(\kappa,\kappa^+) * \dot{\mathbb L}_\mu^h$ name for a subset of $\kappa$ then the transitive closure of (a nice name for) $\dot{x}$ is coded by a set of ordinals of size $\kappa$ and in particular there is a $\gamma < \kappa^+$ so that $\dot{x} \in V[c_i\mid i \neq \gamma]$. 

    Now suppose that $\dot{x}$ is such a name. Let's show that it isn't a pseudo-intersection of $\vec{c}$. Let $\gamma$ be as above for $\dot{x}$. We want to show that $\forces \dot{x} \nsubseteq^* c_\gamma$. Suppose that there is an $\xi < \kappa$ so that $(p, \bar{a}, \dot{F}) \forces \dot{x} \setminus \xi \subseteq c_\gamma$. Now let $y \supseteq p$ be one Cohen generic and let $y' = y^c \cap [{\rm dom}(p), \kappa) \cup p$. This is also Cohen generic. Then $V[y] = V[y']$ so in this model there are the conditions $(\bar{a}, \dot{F}^y)$ and $(\bar{a}, \dot{F}^{y'})$ which are compatible since they have the same stems. Pick a common strengthening $(\bar{b}, \dot{G})$ which, without loss of generality also decides some $\delta \in \dot{x} \setminus (\xi \cup {\rm dom}(p))$. But $\delta$ is only in one of $y$ and $y'$ which is a contradiction.
\end{proof}

In \cite{RaghavanShelah}, Raghavan and Shelah showed that, contrary to the countable case, $\mathfrak{s}_\kappa \leq \mfb_\kappa$ for any regular, uncountable cardinal. Nevertheless this is not the case if we replace $\mfb_\kappa$ by $\mfb_h (\in^*)$.

\begin{lemma}
    Assume $\mathsf{GCH}$. 
    \begin{enumerate}
        \item If $\kappa$ is strongly unfoldable there is a forcing extension in which $\kappa$ is still strongly unfoldable and $\mathfrak{s}_\kappa = \kappa^+ < \mathfrak{b}_h(\in^*) = \mfb_h (\in^{cl}) = \lambda$ for any regular $\lambda > \kappa^+$.
        \item If $\kappa$ is supercompact then there is a forcing extension in which $\mathfrak{b}_h (\in^*) = \mfb_h (\in^{cl}) = \kappa^+ < \mathfrak{s}_\kappa = \lambda$ for any regular $\lambda > \kappa^+$.
    \end{enumerate}
\end{lemma}

\begin{proof}
    We being with Item (1). This is essentially the same as the proof of Lemma \ref{cohen witnesse are preserved}. It is easy to show in the above that the family of Cohen subsets we added will be a splitting family in the final model. The assumption of strong unfoldability is simply needed to ensure, by Theorem \ref{hamkinsindestructibility} that (after preparation) we can assume $\kappa$ remains strongly unfoldable and hence weakly compact in the final model so that we can have $\mathfrak{s}_\kappa > \kappa$ - which will fail otherwise by Suzuki's theorem from \cite{suzuki}. 

   Similarly, if $\kappa$ is indestructibly supercompact then we can force $\mathfrak{p}_\kappa > \kappa^+$ while preserving $\mfb_h (\in^{cl}) = \kappa^+$ for any monotone increasing $h:\kappa \to \kappa$. Therefore, it suffices simply to see the following claim. Recall that by Proposition \ref{Prop: relation between p and s on large} $\mathfrak{s}_\kappa\geq\mathfrak{p}_\kappa$ on a measurable cardinal. 
   
\end{proof}
Putting together all of this, we get the following. 
\begin{theorem}
    If $\kappa$ is inaccessible then $\mfb_h(\in^*)$ and $\mfb_h(\in^{cl})$ are independent of $\mathfrak{p}_\kappa$. If $\kappa$ is supercompact then this is true of $\mathfrak{s}_\kappa$ as well.
\end{theorem}

As before all of these arguments can be intertwined with the lifting arguments from the previous section to obtain the following on a measurable.

\begin{theorem}
    For every $\xi > 1$, relative to a supercompact cardinal $\kappa$ it is consistent that $\kappa$ is measurable and and of the following hold. 
\begin{enumerate}
    \item $\mfb_{\id^{+\xi}}(\in^*) = \mfb_{\id^{+\xi}}(\in^{cl}) < \mathfrak{s}_\kappa = \mathfrak{p}_\kappa$

    \item $\mfb_{\id^{+\xi}}(\in^*) = \mfb_{\id^{+\xi}}(\in^{cl}) > \mathfrak{s}_\kappa = \mathfrak{p}_\kappa$
\end{enumerate}
    In particular, $\mfb_{\id^{+\xi}}(\in^*)$ and $\mfb_{\id^{+\xi}}(\in^{cl})$ are independent of $\mathfrak{p}_\kappa$ and $\mathfrak{s}_\kappa$ and $\mfb_{\id^{+\xi}}(\in^{cl})$ can be strictly less than $\mfb_\kappa (\mathsf{p}\in^*)$ at a measurable.
\end{theorem}

\section{Covering and cofinality correctness of the ultrapowers}\label{Section: Coverling}
Our bounds on $\mathfrak{b}_h(\in^*)$ at the case of a measurable cardinal \ref{cor: trivialities at a measurable} were obtained by the ability to cover the functions below the $([h]_U^+)^{M_U}$. Let us state the exact covering property we can extract from a large $\mfb_h(\in^*)$: 
\begin{proposition}\label{prop: covering}
        Let $h:\kappa\to\kappa$ be any function. If $\mathfrak{b}_h(\in^*)\geq\lambda$, then for every $\sigma$-complete ultrafilter $U$ and set $X\subseteq j_U(\kappa)$, $|X|^V<\lambda$, there is $Y\in M_U$, $M_U\models|Y|\leq|[h]_U|$ such that $X\subseteq Y$.
    \end{proposition}
    \begin{proof}
        For each $x\in X$, set a representing function $f_x$ i.e. $[f_x]_U=x$. Since $|\{f_x\mid x\in X\}|=|X|<\mathfrak{b}_h(\in^*)$, there is an $h$-slalom $\varphi$ such that for every $x\in X$, $f_x\in^*\varphi$. Let $Y=[\varphi]_U\in M_U$. Then $M_U\models |Y|\leq |[h]_U|$ and for every $x\in X$, $x=[f_x]_U\in [\varphi]_U$. 
    \end{proof}
    \begin{corollary}\label{cor: compute cardinality}
         For every ordinal  $\lambda$ with $cf^{M_U}(\lambda)\leq j_U(\kappa)$, if $cf^{V}(\lambda)<\mathfrak{b}_{h}(\in^*)$ then $cf^{M_U}(\lambda)\leq |[h]_U|$
    \end{corollary}
       Suppose that $[h]_U=\kappa^+$ and $\mathfrak{b}_h(\in^*)>\kappa^+$, then the previous corollary says that $M_U$ computes cofinality $\kappa^+$ ordinals correctly as long as they are not greater than $j_U(\kappa)$.
One particularly interesting choice of $\lambda$ is $\lambda=(\kappa^{++})^{M_U}$:
    \begin{corollary}
        Suppose $\kappa^+<\mathfrak{b}_{\id^+}(\in^*)$ then for any normal ultrafilter $U$, $\kappa^{++}=(\kappa^{++})^{M_U}$. 
    \end{corollary}
In relation to Question~\ref{Consistecy}, this corollary puts some restrictions on the kind of embeddings we can expect to lift in order to show that $\kappa$ is a measurable cardinal with $\mfb_{\id^+}(\in^*)>\kappa^+$. To see further restrictions, consider the following example:
    \begin{example}\label{example}
        If we try to force $\mathfrak{b}_{\id^+}(\in^*)>\kappa^+$ from less than supercompactness, we need to create a normal ultrapower which computes $\kappa^{++}$ correctly. One standard way of doing this is due to Woodin (see \cite{CummingsHand}): start with a $(\kappa,\kappa^{++})$-extender $E$ and then lift $j_E$ to an ultrapower embedding in a generic extension which preserves cofinalities. Such ultrapowers are not going to work for a different reason--the ordinal $j_E(\kappa)$. It is inaccessible in $M_U$, but has cofinality $\kappa^+$ in $V$ in contrast with Corollary~\ref{cor: compute cardinality}.
    \end{example}
  \subsection{On the $V$-cofinality of $j_U(\kappa)$}  Example~\ref{example} illustrates a situation where we should consider the cofinality of $j_U(\kappa)$. In this case, if we would like to obtain models where $\mathfrak{b}_{\id^+}(\in^*)>\kappa^+$ we need to be able to have embeddings $cf^V(j_U(\kappa))>\kappa^+$. 
    For example, in the Kunen-Paris Model \cite{Kunen-Paris} or the Friedman-Magidor Model \cite{Friednam-Magidor}, we can even argue that for every ultrafilter $U$, $cf^V(j_U(\kappa))=\kappa^+$ for any normal ultrafilter $U$. The reason is that whenever we force over a model of $GCH$ and lift an  extender/ultrafilter embedding from the ground model, $j_U(\kappa)$ is always going to have cofinality $\kappa^+$.
    We have the following bound for $j_U(\kappa)$:
\begin{proposition}[Folklore]\label{prop: b and d bounds}
    Let $\kappa>\omega$. For every $\sigma$-complete ultrafilter $U$ over $\kappa$, $\mathfrak{b}_\kappa\leq cf^V(j_U(\kappa))\leq \mathfrak{d}_\kappa$.
\end{proposition}
We have uniform bounds for $cf(j_U(\kappa))$, so it is reasonable to conjecture that $cf(j_U(\kappa))$ might not depend on $U$. 
This invites the following interesting question:
\vskip 0.2 cm
\begin{center}
    Can there be two ultrafilters $U,W$ such that $cf^V(j_U(\kappa))\neq cf^V(j_W(\kappa))$?
\end{center}
\vskip 0.2 cm
By Propositions~\ref{prop: b and d bounds},  another scenario in which we get a single cofinality (which is not necessarily $\kappa^+$) is:
\begin{corollary}
    Suppose that $\mathfrak{b}_\kappa=\mathfrak{d}_\kappa=\lambda$, then $cf^V(j_U(\kappa))=\lambda$ for any uniform ultrafilter on $\kappa$.
\end{corollary}
 In particular, this situation occurs if $\mfb_h(\in^*) = 2^\kappa$. In \cite{TomGabeMeasures}, 
 the cardinal invariants $\mathfrak{b}_\kappa,\mathfrak{d}_\kappa$ were used to show that there can be at most one $\lambda$ for which there is a simple $P_\lambda$-point. 
\begin{definition}[Benhamou-Goldberg]
    An ultrafilter $U$ over $\kappa$ is called a simple $\pi P_\lambda$-point if $\pi\mathfrak{p}(U)=\pi\mathfrak{ch}(U)$.
\end{definition}
\begin{corollary}
    If there is a simple $\pi P_\lambda$-point, then for every uniform ultrafilter $U$ on $\kappa$, $cf^V(j_U(\kappa))=\lambda$.
\end{corollary}
The above applies to several models of importance: 
\begin{corollary}
    In the following models, $cf^V(j_U(\kappa))$ is unique and does not depend on the choice of $U$:
    \begin{enumerate}
        \item The extender-based Magidor-Radin models  \cite{TomGabeMeasures}.
        \item The iteration of generalized Mathias forcing \cite{BROOKETAYLOR201737}.
    \end{enumerate}
\end{corollary}
Let us restrict ourselves to model where $\mathfrak{b}_\kappa<\mathfrak{d}_\kappa$. The simplest forcing which does that is the good old Cohen forcing.  M. Canjar \cite{CANJAR19881}, exploited this situation in the Cohen model to obtain ultrafilters on $\omega$ with varying cofinalities. He used the filter extension property  and properties of the Cohen generic reals. If we would like to work below the levels of super and strongly compact cardinals, the filter extension property is not available anymore.  Fortunately, this can be replaced by Woodin's surgery argument, and a later improvement due to Gitik \cite{Gitik1989TheNO}, which stays in the optimal framework of a measurable cardinal with Mitchell order $o(\kappa)=\kappa^{++}$ (see \cite{CummingsHand}):
\begin{theorem}\label{thm:woodin}
    Assume $GCH$ and suppose that there is an elementary embedding $j:V\to M$ such that:
    \begin{enumerate}
        \item $|j(\kappa)|^V=\kappa^{++}$.
        \item there is a function $f:\kappa\to\kappa$ such that $j(f)(\kappa)=\kappa^{++}$.
        \item $M^\kappa\subseteq M$.
    \end{enumerate}
    Then for every $V$-generic\footnote{The forcing $\mathbb{P}_{f}$ is the Easton support iteration of $\text{Add}(\alpha,f(\alpha))$ for every inaccessible $\alpha<\kappa$.}  $G\subseteq\mathbb{P}_{f}*\text{Add}(\kappa,\kappa^{++})$, $j$ lifts to an ultrapower by a normal $\kappa$-complete ultrafilter $U$ over $\kappa$. In particular, $j_U(\kappa)=j(\kappa)$.
\end{theorem}
These kinds of embeddings were used in \cite{CichonMaximum} and in \cite{GitikCich}
\begin{theorem}[Gitik {\cite{Gitik1989TheNO}}]\label{thm:Gitik}
Assume that there is a coherent sequence $\vec{U}$ of length $\kappa^{++}$. Then it is consistent that there is a forcing extension where GCH holds and there is a commutative\footnote{A Rudin-Keisler increasing sequence is commutative, if the Rudin-Keisler projections $\pi_{\alpha,\beta}$ commute. Namely for every $\alpha<\beta<\gamma$, $\pi_{\alpha,\gamma}=\pi_{\beta,\gamma}\circ \pi_{\alpha,\beta}$.} Rudin-Keisler increasing sequence $\l U_i\mid i<\kappa^{++}\r$.
\end{theorem}
\begin{corollary}
    Relative to a measurable cardinal $\kappa$ with $o(\kappa)=\kappa^{++}$, it is consistent that there are two $\kappa$-complete ultrafilters $U,W$ such that $cf^V(j_U(\kappa))\neq cf^V(j_W(\kappa))$. 
\end{corollary}
\begin{proof}
    Work in the generic extension $V_0$ of Theorem~\ref{thm:Gitik}. Let us produce two  embeddings $j^{1}:V_0\to M^{(1)}$ and $j^{(2)}:V_0\to M^{(2)}$ as in Theorem~\ref{thm:woodin} each witness a different cofinality of $j^{(i)}(\kappa)$.
    Indeed, we let $j_{U_i}:V_0\to M_{U_i}$ be the ultrapower embedding of $V_0$ by $U_i$. 
    For $i<j$ there is a factor map $k_{i,j}:M_{U_i}\to M_{U_j}$ defined using the Ruding Keisler projection $\pi_{i,j}$ of $U_j$ to $U_i$ by $k_{i,j}([f]_{U_i})=[f\circ \pi_{i,j}]_U$. It follows that $j_{U_j}=k_{i,j}\circ j_{U_i}$. It is not hard to check that the commutativity of the system $\vec{U}=\l U_i\mid i<\kappa^{++}\r$ ensures that the functions $k_{i,j}$ commute and therefore we may take a direct limit $M_{\infty}=\text{dir}\lim \l M_{U_i},k_{i,j}\mid i<j<\kappa^{++}\r$.
    Let $j_{\infty}:V_0\to M_{\infty}$ be the direct limit embedding of the system $\l j_{U_i},k_{i,j}\mid i<j<\kappa^{++}\r$. By the GCH in $V_0$, it is not hard to see that $\kappa_1:=j_{\infty}(\kappa)=\kappa^{++}$, and that $M_{\infty}^{\kappa}\subseteq M_{\infty}$, but there cannot be a function $f:\kappa\to\kappa$ such that $j_{\infty}(f)(\kappa)=\kappa^{++}=j_{\infty}(\kappa)$. To overcome this problem, we follow the proof from \cite[\S2, p.212]{Gitik1989TheNO}, we may assume that $V_0$ is a generic extension of a ground model $V$ by an Easton support iteration of adding one Cohen function $f_\alpha:\alpha\to\alpha$ for every inaccessible $\alpha\leq\kappa$. Suppose $V_0=V[H][f_\kappa]$, where $H$ is $V$-generic for the iteration up to (not including) $\kappa$ and $f_\kappa$ is $V[H]$-generic for $\text{Add}(\kappa,a)$.  
to obtain the embedding $j^{(1)}$, take again the direct limit by $j_{\infty}(\vec{U})$ starting from $M_{\infty}$. We obtain $j':M_{\infty}\to M^{(1)}$ with $\kappa_2:=j'(\kappa_1)=(\kappa_1^{++})^{M_{\infty}}>\kappa^{++}$. By elementarity of $j'\circ j$, note that $M^{(1)}$ is a generic extension of a ground model $M_*=j'(j_{\infty}(V))$, namely  $M^{(1)}=M_*[j'(j_{\infty}(H))][f'_{\kappa_2}]$. Let us change one value in $f'_{\kappa_2}:\kappa_2\to \kappa_2$, by setting $f_{\kappa_2}(\kappa)=\kappa^{++}$ and $f_{\kappa_2}(\alpha)=f'
_{\kappa_2}(\alpha)$ elsewhere. Then clearly, $f_{\kappa_2}$ remains $M_*[j'(j_{\infty}(H))]$-generic, and  $M^{(1)}=M_*[j'(j_{\infty}(H))][f'_{\kappa_2}]=M_*[j'(j_{\infty}(H))][f_{\kappa_2}]$.

Finally, we let $j^{(1)}:V_0\to M^{(1)}$ be defined by $j\restriction V=j'\circ j_{\infty}\restriction V$ and $j(H)=j'(j_{\infty}(H))$ and $j^{(1)}(f_\kappa)=f_{\kappa_2}$. To summarize, we have that:
\begin{enumerate}
    \item $cf^V(j^{(1)}(\kappa))=cf^V(j_{\infty}(\kappa^{++}))=\kappa^{++}$.
    \item $j^{(1)}(f_\kappa)(\kappa)=\kappa^{++}$.
    \item $(M^{(1)})^\kappa\subseteq M^{(1)}$.
\end{enumerate}
For the embedding $j^{(2)}$, we do something similar to the above which is exactly as in \cite{Gitik1989TheNO},  applying the ultrapower by $j_{\infty}(U_0)$ from $j_{j_{\infty}(U_0)}:M_{\infty}\to M^{(2)}$ and changing the value of $j_{j_{\infty}(U_0)}(j_{\infty}(f_\kappa))(\kappa)$ to be $\kappa^{++}$. This produces an embedding $j^{(2)}:V\to M^{(2)}$ such that $j^{(2)}(\kappa)=j_{j_{\infty}(U_0)}(j_{\infty}(\kappa))$.
We conclude that $j^{(2)}$ satisfy:
\begin{enumerate}
    \item $cf^{V}(j^{(2)}(\kappa))=cf^V(j_{\infty}(\kappa^+))=\kappa^+$.
    \item $j^{(2)}(f_\kappa)(\kappa)=\kappa^{++}$.
    \item $(M^{(2)})^\kappa\subseteq M^{(2)}$.
\end{enumerate}
We may now apply Theorem~\ref{thm:woodin} to obtain that in the generic extension by $\mathbb{P}_{f_\kappa}*\text{Add}(\kappa,\kappa^{++})$, we may lift $j^{(1)},j^{(2)}$ to normal ultrapower embedding $j_{U^{(1)}},j_{U^{(2)}}$ respectively, and $cf^V(j_{U^{(1)}}(\kappa))=\kappa^{++}$ while $cf^V(j_{U^{(2)}}(\kappa))=\kappa^{+}$, as wanted.
\end{proof}
\bibliographystyle{amsplain}
\bibliography{ref}
\end{document}